\documentclass[11pt]{amsart}
\usepackage[utf8]{inputenc}
\usepackage{amsmath}
\usepackage{amssymb}
\usepackage{graphicx, overpic}
\usepackage{caption}
\usepackage{subcaption}
\usepackage{pinlabel}
\usepackage{tikz}
\usepackage{eucal}
\usepackage[all]{xy}

\SelectTips{cm}{12}

\usepackage{color}

\usepackage[margin=1.35in]{geometry}

\numberwithin{equation}{section}
\newtheorem{thm}[equation]{Theorem}
\newtheorem{prop}[equation]{Proposition}
\newtheorem{lem}[equation]{Lemma}
\newtheorem{cor}[equation]{Corollary}

\theoremstyle{definition}
\newtheorem{defn}[equation]{Definition}
\newtheorem{rem}[equation]{Remark}

\newtheorem{exmp}[equation]{Example}
\newtheorem{ques}[equation]{Question}
\newtheorem{prob}[equation]{Problem}
\newtheorem{cons}[equation]{Construction}
\newtheorem{notation}[equation]{Notation}
\newtheorem{assumption}[equation]{Standing Assumptions}

\newcommand{\field}[1]{\mathbb{#1}}
\newcommand{\Z}{\field{Z}}

\newcommand{\R}{\field{R}}

\newcommand{\E}{\field{E}}
\newcommand{\Hyp}{\field{H}}

\newcommand{\cE}{\mathcal{E}}

\newcommand{\cG}{\mathcal{G}}

\newcommand{\la}{\left\langle}
\newcommand{\ra}{\right\rangle}
\newcommand{\p}{\partial}

\newcommand{\of}{\circ}
\newcommand{\boundary}{\partial}

\newcommand{\bigset}[2]{ \bigl\{ \, {#1} \bigm| {#2} \, \bigr\} }

\renewcommand{\emptyset}{\varnothing}

\DeclareMathOperator{\Isom}{Isom}
\DeclareMathOperator{\CAT}{CAT}

\DeclareMathOperator{\Stab}{Stab}

\DeclareMathOperator{\Circ}{Circ}

\definecolor{amethyst}{rgb}{0.6, 0.4, 0.8}

\newcommand{\hide}[1]{}

\begin{document}

\title{Coarse Alexander duality for pairs and applications}
\author{G. Christopher Hruska}
\address{Department of Mathematical Sciences\\
University of Wisconsin--Milwaukee\\
PO Box 413\\
Milwaukee, WI 53211\\
USA}
\email{chruska@uwm.edu}

\author{Emily Stark}
\address{Department of Mathematics and Computer Science, Wesleyan University, 265 Church Street, Middletown, CT 06457, USA}
\email{estark@wesleyan.edu}

\author{Hung Cong Tran}
\address{FCCI Insurance Group, 6300 University Parkway, Sarasota, FL 34240, USA}
\email{htran@fcci-group.com}



\date{\today}

\begin{abstract}
For a group $G$ (of type $F$) acting properly on a coarse Poincar\'{e} duality space $X$, Kapovich--Kleiner introduced a coarse version of Alexander duality between~$G$ and its complement in $X$.
More precisely, the cohomology of $G$ with group ring coefficients is dual to a certain \v{C}ech homology group of the family of increasing neighborhoods of a $G$--orbit in $X$.
This duality applies more generally to coarse embeddings of certain contractible simplicial complexes into coarse $PD(n)$ spaces.
In this paper we introduce a relative version of this \v{C}ech homology that satisfies the Eilenberg--Steenrod Exactness Axiom, and we prove a relative version of coarse Alexander duality.

As an application we provide a detailed proof of the following result, first stated by Kapovich--Kleiner.  
Given a $2$--complex formed by gluing~$k$ halfplanes along their boundary lines and a coarse embedding into a contractible $3$--manifold, the complement consists of~$k$ deep components that are arranged cyclically in a pattern called a \emph{Jordan cycle}.
We use the Jordan cycle as an invariant in proving the existence of a $3$--manifold group that is virtually Kleinian but not itself Kleinian.
\end{abstract}

\maketitle

\section{Introduction}

The coarse Alexander duality theorem of Kapovich--Kleiner~\cite{kapovichkleiner05} is a powerful homological tool for the study of coarse embeddings via the asymptotic shapes of their complements, generalizing earlier work of \cite{FarbSchwartz96,BlockWeinberger97}.  In certain circumstances, a coarse embedding of simplicial complexes $f\colon K \to X$ induces a duality between the coarse cohomology of $K$ and a certain \v{C}ech homology group associated to the ``end'' of the complement $X\setminus f(K)$, which, roughly speaking, measures the homology of cycles and boundaries in $X$ that are ``far from~$f(K)$.''

Coarse cohomology was introduced by Roe, who shows in \cite[Chap.~5]{Roe_CoarseGeometry} that for a uniformly acyclic simplicial complex $K$, the coarse cohomology of $K$ with integral coefficients coincides with the simplicial cohomology with compact supports $H_c^*(K)$.  Furthermore, in the presence of a proper, cocompact action of a group $G$ on $K$, it also coincides with the group cohomology $H^*(G;\Z G)$ of $G$ with coefficients in the group ring $\Z G$. (See also \cite[\S VIII.7]{Brown_Cohomology}.)

The \emph{end} of $X$ is the inverse sequence of spaces $\epsilon X = \{X \setminus N_i\}$, where $N_i$ is the closed simplicial ball in $X$ of radius $i$ centered at a fixed basepoint in $X$ and the bonding maps are inclusions.
The \v{C}ech homology of the end of $X$ is the inverse limit of the integral homology groups $H_*(\overline{X \setminus N_i})$.
In a sense, \v{C}ech homology of the end detects the shape of a space near infinity (see \cite{Geoghegan86_ShapeOfAGroup,guilbault,GeogheganSwenson_Semistable}).
The increasing family of balls may be viewed as increasing neighborhoods of a point in $X$.
If, instead of balls, one chooses $\{N_i\}$ to be a sequence of increasing tubular neighborhoods of a subcomplex $Y\subset X$, one obtains the \emph{$Y$--end} of $X$ and the \emph{\v{C}ech homology of the $Y$--end}, denoted $\check{H}^D_{k}(Y^c)$.
In this notation $Y^c$ denotes the complement $X \setminus Y$.  Informally, this \v{C}ech homology group may also be called \emph{deep homology of the complement} and its elements are informally called \emph{deep homology classes}.

A coarse $PD(n)$ space, or coarse Poincar\'{e}\ duality space of formal dimension $n$, is a type of space introduced by Kapovich--Kleiner \cite{kapovichkleiner05} that has many large scale homological features in common with contractible $n$--manifolds (see Section~\ref{sec:rel_homology}).  For example, coarse $PD(n)$ spaces include the universal covers of closed aspherical $n$--dimensional PL manifolds or any acyclic simplicial complex that admits a free cocompact action by a Poincar\'{e}\ duality group of dimension $n$.

Suppose $X$ is a coarse $PD(n)$ space, and $Z$ is a uniformly acyclic simplicial complex. In \cite{kapovichkleiner05},
Kapovich--Kleiner show that a coarse embedding $f\colon Z \to X$ induces a duality isomorphism between 
the \v{C}ech reduced homology $\check{\widetilde{H}}{}^D_k (fZ^c)$ of the $f(Z)$--end of $X$ and the compactly supported cohomology of $Z$.
This coarse Alexander duality generalizes the classical Alexander duality for embeddings of compacta in the $n$--sphere
to a coarse setting suitable for the study of coarse embeddings between spaces or groups.

In order to facilitate computations in a homology theory, it is useful to introduce relative homology groups and study the corresponding homology long exact sequences.
In this article, we introduce a relative version of deep homology of the complement.
A simplicial pair $(Z_0,Z_1)$ is a \emph{coarse simplicial pair} if $Z_0$ and $Z_1$ are each endowed with a length metric in which each edge has length one, and the inclusion $Z_1\to Z_0$ is a coarse embedding.
Given a coarse simplicial pair $(Z_0,Z_1)$, a coarse embedding $f\colon Z_0 \to X$ induces a coarse embedding $f\colon Z_1 \to X$.
In this context, we introduce the \emph{\v{C}ech relative homology of the $\bigl(f(Z_0),f(Z_1)\bigr)$--end}, denoted $\check{H}^D_k \bigl( fZ_1^c,fZ_0^{c} \bigr)$.

A well-known difficulty with the notion of \v{C}ech homology is that it typically does not satisfy the Eilenberg--Steenrod Exactness Axiom, which requires that the homology sequence for any pair be exact \cite{EilenbergSteenrod_Axioms,Milnor_Steenrod}.
The main source of this difficulty is that \v{C}ech homology theory is based on inverse limits, but the limit of an inverse system of exact sequences need not be exact.
For this reason, the standard exact sequence techniques of classical algebraic topology are severely limited when working with \v{C}ech homology.
However, we prove in the setting of coarse embeddings into coarse $PD(n)$ spaces, the \v{C}ech homology groups as above do satisfy the Exactness Axiom.

\begin{thm}[Deep homology is exact]
\label{thm:Exact}
Let $(Z_0,Z_1)$ be a coarse simplicial pair of uniformly acyclic complexes, each with bounded geometry.
Consider a coarse $PD(n)$ space $X$ and a coarse embedding $f\colon Z_0\to X$.
Then the following \v{C}ech homology sequence is exact:
\[
   \cdots \longrightarrow 
   \check{H}_{n-k}^D(fZ_1^c,fZ_0^{c}) \xrightarrow[\quad]{\boundary}  
   \check{\widetilde{H}}{}^D_{n-k-1}(fZ_0^{c}) \longrightarrow
   \check{\widetilde{H}}{}^D_{n-k-1}(fZ_1^{c}) \longrightarrow
   \cdots
\]
\end{thm}

We refer to Section~\ref{sec:CoarseGeometry} for precise definitions of uniformly acyclic, bounded geometry, and coarse embeddings.

In particular, as explained by Milnor in \cite{Milnor_Steenrod}, exactness implies that these \v{C}ech homology groups coincide with the corresponding Steenrod homology groups (equivalently, the filtered locally finite homology groups) of the $f(K)$--end as defined in \cite[\S 14.2]{geoghegan}.
A key technical step in the proof of Theorem~\ref{thm:Exact} is provided by a result of Dydak \cite{dydak} analogous to the five lemma, but in the setting of exact sequences of inverse sequences (see Lemma~\ref{lem:FiveLemma}).

Generalizing Kapovich--Kleiner \cite{kapovichkleiner05}, we show there is a duality isomorphism between compactly supported relative cohomology and relative \v{C}ech homology of the complement, which induces an isomorphism of long exact sequences in the following sense.

\begin{thm}[Coarse Alexander Duality of Pairs]
\label{thm:CAD_Pairs_Main}
Let $(Z_0,Z_1)$ be a coarse simplicial pair of uniformly acyclic complexes, each with bounded geometry.
Consider a coarse $PD(n)$ space $X$ and a coarse embedding $f\colon Z_0\to X$.
There exists a commutative diagram 
\[
\xymatrix@C-6pt{
   \cdots \ar[r] &
   \displaystyle
   \check{H}_{n-k}^D(fZ_1^c,fZ_0^{c}) \ar[r]^{\,\,\,\,\,\boundary} \ar[d]^{A}_{\cong} & 
   \displaystyle \check{\widetilde{H}}{}^D_{n-k-1}(fZ_0^{c}) \ar[r] \ar[d]^{A}_{\cong}  &
   \displaystyle \check{\widetilde{H}}{}^D_{n-k-1}(fZ_1^{c})  \ar[r] \ar[d]^{A}_{\cong}  &
   \cdots \\
   \cdots \ar[r]^{\!\!\!\!\!\!\!\!\!\!\!\!\!\delta} &
   H_c^k(Z_0,Z_1) \ar[r] &
   H_c^k(Z_0)   \ar[r] &
   H_c^k(Z_1)    \ar[r]^{\ \ \ \delta} &
   \cdots
}
\]
in which the vertical maps are coarse Alexander duality isomorphisms, and the rows are homology and cohomology long exact sequences.
\end{thm}

\subsection{Jordan separation and Jordan cycles}
The classical Alexander duality theorem proves for any embedding $S^{n-2} \to S^n$, the complement $S^n \setminus S^{n-2}$ always has the same reduced homology as $S^1$, an infinite cyclic group concentrated in dimension one.  
A generator of this cyclic homology group is called a \emph{Jordan cycle}, by analogy with the codimension-one case of Jordan--Brouwer separation.
For a related example, consider the join $\Theta_k^{n-1}$ of $S^{n-2}$ with a discrete set of $k$ points. 
According to Alexander duality, any subspace of $S^n$ homeomorphic to $\Theta_k^{n-1}$ separates the $n$--sphere into $k$ components arranged cyclically in a pattern corresponding to the Jordan cycle.
To make precise the similarity between the Jordan cycle and the adjacency of components of $S^n \setminus \Theta_k^{n-1}$, one may use the long exact sequence of a pair as a fundamental computational tool.

The main objects of study in coarse Alexander duality are complexes coarsely embedded in a coarse $PD(n+1)$ space $X$.  To keep the analogy with the classical setting, we imagine that $X$ is compactified by a fictitious sphere $S^n$ at infinity (a more realistic fantasy might involve a homology $n$--manifold that is a homology $n$--sphere as in \cite{Bestvina_LocalHomology}) and that coarse embeddings in $X$ correspond to topological embeddings in $S^n$.

We demonstrate the use of relative coarse Alexander duality by formalizing the notion of a Jordan cycle in the coarse setting, a notion that is implicit in work of Kapovich--Kleiner~\cite{kapovichkleiner05}. 
In \cite{kapovichkleiner05} and \cite{hruskastarktran}, this notion is used to prove that certain groups are not $3$--manifold groups.

By analogy with the embeddings $\Theta_k^{n-1} \rightarrow S^n$, we consider coarse embeddings into a coarse $PD(n+1)$ space of a union of $k$ half-spaces $\R^{n}_{\geq 0}$ glued together along their boundaries to form a space $W_k$. For $n=2$, these spaces arise naturally in Cayley $2$--complexes for Baumslag--Solitar groups and certain other amalgams of surface groups; see \cite{kapovichkleiner05, hruskastarktran} and Section~\ref{sec:comm}. The coarse Alexander duality theorem immediately implies that a coarse embedding of $W_k$ into a coarse $PD(n+1)$ space $X$ coarsely separates it into $k$ deep components corresponding to the set of relative ends $\cE\bigl(X, f(W_k)\bigr)$. 
    
We use $1$--dimensional relative deep homology groups and the long exact sequence of a pair to define when two deep components of $\cE\bigl(X, f(W_k)\bigr)$ are adjacent. More specifically, a pair of deep components is \emph{adjacent} if the image of the connecting homomorphism $\partial\colon \check{H}_1^D\bigl((fW_{k-1}^i)^c, fW_k^c \big) \rightarrow \check{H}_0^D(fW_k^c)$ for some $i$ is supported on the pair of components, where $W_{k-1}^i \subset W_k$ is obtained by removing the $i^{th}$ halfspace from $W_k$. These notions define an adjacency graph $\Gamma_k$, with vertices corresponding to deep components, and an edge if two components are adjacent.
    
The coarse Jordan cycle, a $1$--dimensional deep homology class, is then used to describe the adjacency relations and the structure of the graph $\Gamma_k$. Let $W_0 \subset W_k$ denote the image of the boundaries of the halfspaces in $W_k$ after gluing. Coarse Alexander duality implies that the $1$--dimensional homology $\check{H}_1^D(fW_0^c)$ is infinite cyclic. We call a generator of this homology group a \emph{Jordan cycle}.
Using relative homology, we recover a proof of the following result that extends a theorem of Kapovich--Kleiner \cite[Lem.~7.11]{kapovichkleiner05}.

\begin{thm}
\label{thm:AdjacencyGraph}
\textup{(}Theorem~\ref{thm:map_of_adj_graphs}.\textup{)}
The adjacency graph $\Gamma_k$ is a circuit. 
If $W_{k-1}^i$ is obtained from $W_k$ by removing the $i$th halfspace, the inclusion $W_{k-1}^i \to W_k$ induces a natural map of cycles $\Gamma_k \to \Gamma_{k-1}^i$ given by collapsing a single edge to a vertex.
\end{thm}

Although the extra details in the above statement are implicit in the proof of \cite[Lem.~7.11]{kapovichkleiner05}, we make this conclusion more explicit using the general methods of this paper. 
These extra details are required to complete the proof of Theorem~5.6 in \cite{hruskastarktran}.
We also use these extra details in an essential way in the proof of Theorem~\ref{thm:SurfaceAmalgam}, discussed below.
We suspect that these general methods may prove useful more broadly for studying the relationship between coarse embeddings of groups and coarse embeddings of their various subgroups.

\subsection{Applications to $3$--manifolds}

A \emph{Kleinian group} is a discrete group of isometries of $\Hyp^3$.  
Torsion-free groups that are virtually Kleinian but that are not the fundamental group of any $3$--manifold are constructed in \cite{kapovichkleiner} and \cite{hruskastarktran}.
Section~\ref{sec:comm} is an extended study of one example that arises in \cite{hruskastarktran} for which a more subtle virtual geometric property arises.
We show that, among the family of fundamental groups of $3$--manifolds, those that admit Kleinian representations are not closed under commensurability.

The fundamental group of the $2$--complex formed by gluing a closed orientable hyperbolic surface to a torus along an essential simple closed curve in each surface is a geometrically finite Kleinian group (see Proposition~\ref{prop:Accidental}).
In contrast, we show that a certain amalgam of a closed hyperbolic surface group with a Klein bottle group over a $\Z$ subgroup does not admit a Kleinian representation, even though it is a $3$--manifold group and it is virtually Kleinian.

\begin{thm}
\label{thm:SurfaceAmalgam}
Let $X$ be the space obtained from a closed orientable surface of negative Euler characteristic and a Klein bottle by gluing a nonseparating simple closed curve on the surface to an orientation reversing simple closed curve on the Klein bottle. Then $G=\pi_1(X)$ satisfies the following properties. 
\begin{enumerate}
    \item
    \label{item:TwistedGluing}
    $G$ is the fundamental group of a $3$--manifold.
    \item
    \label{item:Hyperbolization}
    $G$ has an index--$2$ subgroup that is isomorphic to a Kleinian group. 
    \item 
    \label{item:NotKleinian}
    $G$ is not isomorphic to a Kleinian group. 
    \end{enumerate}
\end{thm}
  
Statement~(\ref{item:TwistedGluing}) of Theorem~\ref{thm:SurfaceAmalgam} is proved by a direct construction in \cite{hruskastarktran}. Statement~(\ref{item:Hyperbolization}) is a simple application of Thurston's Hyperbolization Theorem (see Proposition~\ref{prop:Accidental}).
The proof of Statement~(\ref{item:NotKleinian}) uses Jordan cycles as a tool.  We briefly sketch this proof here.  In order to convert this sketch into a proof, we use the naturality properties of Jordan cycles that are established in Theorem~\ref{thm:AdjacencyGraph}.
We refer the reader to Section~\ref{sec:comm} for more details.
The universal cover of $X$ contains the union of a Euclidean plane and a hyperbolic plane that intersect in a line.  The stabilizer $H$ of this line acts on the corresponding adjacency graph given by Theorem~\ref{thm:AdjacencyGraph}.
Since this graph is a circuit of length four, its symmetry group is dihedral.
As the induced action of $H$ on the cycle contains a transposition of the Euclidean halfplanes, we conclude that the Euclidean and hyperbolic planes must intersect transversely (in a coarse sense).
On the other hand, in any Kleinian representation, the Euclidean plane must coarsely map onto a horosphere, giving a contradiction.
We conclude that the virtually Kleinian groups given by Theorem~\ref{thm:SurfaceAmalgam} are $3$--manifold groups, but any such $3$--manifold must have a nontrivial geometric decomposition.

Among the family of irreducible $3$--manifolds (possibly with boundary) with zero Euler characteristic,
the property of admitting a geometric structure (in the sense of Thurston) is a quasi-isometry invariant of the fundamental group by work of Kapovich--Leeb~\cite{KapovichLeeb_QI} and the Geometrization Theorem. 
In contrast, among compact, irreducible  $3$--manifolds whose boundary has negative Euler characteristic, the property of admitting a geometric structure is not a quasi-isometry invariant of the fundamental group.  Indeed Theorem~\ref{thm:SurfaceAmalgam} illustrates the existence of two compact $3$--manifolds with the same fundamental group such that one is geometric and the other is not.
On the other hand, we note that the canonical geometric decomposition along tori and annuli as defined, for example, in \cite{NeumannSwarup_Canonical} is nontrivial for these manifolds, even though one admits a hyperbolic structure.

\subsection{Potential directions for future study}

The \emph{action dimension} of a discrete group~$G$ is the minimum dimension of a contractible manifold $M$ that admits a proper action of the group~$G$ (see Bestvina--Kapovich--Kleiner \cite{bestvinakapovichkleiner}).
If $G$ is torsion free, then the aspherical manifold $M/G$ is a finite-dimensional model for the classifying space $BG$, so the action dimension is greater than or equal to the geometric dimension, \emph{i.e.}, the minimal dimension of a $K(G,1)$ CW complex.
In \cite{hruskastarktran}, the authors use coarse Alexander duality to show that many amalgams of surface groups over cyclic groups have action dimension four, generalizing examples due to Kapovich--Kleiner \cite{kapovichkleiner05}.

Broadly speaking, the concepts of coarse Alexander duality and its relative version may be useful in the study of actions of groups on contractible $n$--manifolds $M$, in terms of restricting the dimension $n$ (as in \cite{kapovichkleiner05, hruskastarktran}) or restricting possible geometric structures on $M$ (as in Theorem~\ref{thm:SurfaceAmalgam}).
Obstructions to coarse embedding (in particular, the embedding induced by a proper action) are often related to the shape of the complement far from the orbit of the group. Using coarse Alexander duality, one can detect the shape of the complement via the relative \v{C}ech homology of the complement. One may be able to use this information as a tool to prove that a group $G$ coarsely embeds into a contractible $n$--manifold $M$ but cannot act properly on $M$.  

Coarse Alexander duality and its relative version could potentially be used to study the geometric relationship between a finitely generated group $G$ and its subgroups. More precisely, we highlight the following question for which coarse Alexander duality may be useful.

\begin{prob}
Classify subgroups of $G$ up to equivalence under the action of the quasi-isometry group $QI(G)$.
More precisely, given two subgroups $H$ and $K$ of a finitely generated group $G$, determine whether there exists a quasi-isometry of $G$ that maps $H$ to $K$ within a bounded distance.  
\end{prob}

Study of the above problem may illuminate the general study of the quasi-isometry group of $G$ in the presence of ``pattern rigidity'' results. In the literature, various quasi-isometry invariants of a pair $(G,H)$ with $H\le G$ have been studied, such as subgroup distortion and its higher dimensional variants, relative divergence, and many others, but they have limited use. Using coarse Alexander duality, one may be able to use the \v{C}ech homology of the complement of $H\le G$ to better understand the large-scale geometry of~$G$. 

Suppose a Poincar\'{e} duality group $G$ of dimension $n$ acts freely, cocompactly, and simplicially on a contractible coarse $PD(n)$ space $X$.  
The following is a special case of a conjecture due to Bestvina on $\mathcal{Z}$--set compactifications:
Does $G$ admit a  $\mathcal{Z}$--structure boundary in the sense of \cite{Bestvina_LocalHomology}?
If so, then a theorem of Bestvina states that (in many cases) such a boundary must be a homology manifold that is also a homology sphere \cite[Thm.~2.8]{Bestvina_LocalHomology}.
The existence of such a boundary at infinity would provide a more concrete interpretation of the ends of filtered spaces (and of pairs of filtrations) studied in this paper as direct analogues of classical Alexander duality of subsets (and pairs of subsets) of homology spheres.

\subsection{Organization of the paper}

Section~\ref{sec:Stability} contains a review of inverse systems and stability properties.
In Section~\ref{sec:CoarseGeometry}, we review background on coarse geometry and applications to the study of chain maps.
Section~\ref{sec:DeepHomology} introduces \v{C}ech homology of complements in ordinary and relative forms and discusses the relation with the notion of ends of pairs of spaces.
In Section~\ref{sec:CoarseAlexanderDuality}, we review the construction of coarse Alexander duality due to Kapovich--Kleiner.  We establish terminology and notation that are useful for introducing the relative version in Section~\ref{sec:DualityForPairs}.
Theorems \ref{thm:Exact}~and~\ref{thm:CAD_Pairs_Main} are immediate consequences of Theorem~\ref{thm:CAD_Pairs}.
Section~\ref{sec:Jordan_adj_graph} contains the study of Jordan cycles and the proof of Theorem~\ref{thm:AdjacencyGraph}.
An application to $3$--dimensional manifolds is discussed in Section~\ref{sec:comm}, which includes the proof of Theorem~\ref{thm:SurfaceAmalgam}.

\subsection{Acknowledgments}
The authors are thankful for helpful discussions with Arka Banerjee, Max Forester, Ross Geoghegan, Craig Guilbault, Joseph Maher, Alexander Margolis, Boris Okun, Jing Tao, and Genevieve Walsh.
We are grateful for the feedback of the referee, whose comments helped to improve the exposition of the paper.
The first author is partially supported by grants \#318815  and \#714338 from the Simons Foundation. The second author was supported by the Azrieli Foundation; was partially supported at the Technion by a Zuckerman Fellowship; was supported by NSF RTG grant $\#$1840190; and was supported by NSF Grant No. DMS-2204339. The third author was partially supported by an AMS-Simons Travel Grant.

\section{Stability of inverse sequences}
\label{sec:Stability}

In this section, we review some definitions and useful properties of inverse systems and inverse sequences.
For detailed background, we refer the reader to either Marde\v{s}i\'{c}--Segal \cite{MardesicSegal82} or Geoghegan \cite{geoghegan}.
In this paper we mainly treat inverse systems using the category of pro-groups.
In fact, our main results deal only with inverse systems of abelian groups, so they could be expressed in the category of pro-(Abelian groups).  In the language of category theory, we note that the inclusion functor from the pro-Abelian category to the pro-group category is a fully faithful embedding.
We work with pro-groups in order to keep statements simpler and more general whenever possible.

We note that in \cite{kapovichkleiner05} Kapovich--Kleiner work in a related category, which they call \emph{Approx}.
As explained in \cite[\S 4]{kapovichkleiner05}, their main theorems dealing with inverse sequences in the \emph{Approx} category can all be expressed in the language of pro-(Abelian groups).  Throughout this paper we apply this translation without comment whenever citing their results.

\begin{defn}[Inverse system]
A \emph{directed set} is a partially ordered set $\mathcal{A}$ such that for each $\alpha_0,\alpha_1 \in \mathcal{A}$ there exists $\alpha \in \mathcal{A}$ with $\alpha_0\le \alpha$ and $\alpha_1 \le \alpha$.

An \emph{inverse system} $(G_\alpha,p_\alpha^{\alpha'})$ consists of a family of groups $(G_\alpha)_{\alpha\in\mathcal{A}}$ indexed by a directed set $\mathcal{A}$ and a collection of homomorphisms $p_{\alpha}^{\alpha'} \colon G_{\alpha'} \to G_\alpha$ for $\alpha\le\alpha'$ in $\mathcal{A}$ satisfying the following rules:
\begin{itemize}
    \item $p_{\alpha}^{\alpha'} p_{\alpha'}^{\alpha''}=p_{\alpha}^{\alpha''}$ whenever $\alpha\le\alpha'\le\alpha''$, and
    \item $p_\alpha^\alpha$ is the identity homomorphism for each $\alpha$.
\end{itemize}
The groups $G_\alpha$ are the \emph{terms}, the maps $p_\alpha^{\alpha'}$ are the \emph{bonding maps}, and $\mathcal{A}$ is the \emph{index set} of the inverse system.
When the bonding maps and index set are clear from context, the inverse system $(G_\alpha,p_\alpha^{\alpha'})$  may be denoted by $(G_\alpha)_{\alpha\in \mathcal{A}}$ or $(G_\alpha)_\alpha$ or $(G_\alpha)$.

An inverse system indexed by the natural numbers is an \emph{inverse sequence}.  An inverse system indexed by a singleton is a \emph{rudimentary system} and is denoted by $(G)$, where $G$ is its only term.
\end{defn}

\begin{defn}[Morphisms]
A \emph{morphism of inverse systems}
\[
   f \colon (G_\alpha)_{\alpha \in \mathcal{A}} \to (H_\beta)_{\beta\in\mathcal{B}}
\]
consists of an increasing function $\phi\colon \mathcal{B} \to \mathcal{A}$ called the \emph{index function} and for each $\beta\in \mathcal{B}$ a homomorphism $f_\beta\colon G_{\phi(\beta)} \to H_\beta$ such that whenever $\beta \le \beta'$
the following diagram commutes:
\[
\xymatrix{
   G_{\phi(\beta)} \ar[d]_{f_\beta} & G_{\phi(\beta')} \ar[d]^{f_{\beta'}} \ar[l]  \\
   H_\beta & H_{\beta'} \ar[l]
}
\]

For simplicity, we suppress the labels of the bonding maps between terms of the same inverse system.
The notion of composition of morphisms is defined in the obvious way.

A \emph{level morphism} between inverse systems indexed by the same set $\mathcal{A}$ is a morphism whose index function $\phi\colon \mathcal{A}\to\mathcal{A}$ is the identity function.
The \emph{identity morphism} $(G_\alpha)\to(G_\alpha)$ is the level morphism consisting of the family of identity maps $G_\alpha\to G_\alpha$ for each $\alpha\in \mathcal{A}$.

Two morphisms $f$ and $f'$ of inverse systems
$(G_\alpha)_{\alpha\in\mathcal{A}}\to (H_\beta)_{\beta\in\mathcal{B}}$
with index functions $\phi$ and $\phi'$ are \emph{equivalent} if for each $\beta\in\mathcal{B}$ there exists $\alpha \in \mathcal{A}$ with $\alpha \ge \phi(\beta)$ and $\alpha \ge\phi'(\beta)$ such that the following diagram commutes:
\[
\xymatrix{
   G_{\phi(\beta)} \ar[d]_{f_\beta} & G_{\alpha} \ar[d]\ar[l]  \\
   H_\beta & G_{\phi'(\beta)} \ar[l]_{f'_\beta}
}
\]
\end{defn}

\begin{defn}[Pro-groups]
The category of \emph{pro-groups} is defined as follows:
the objects are inverse systems of groups, and the morphisms $[f]$ are equivalence classes of morphisms $f$ between inverse systems.
\end{defn}

\begin{defn}
A morphism of inverse systems $f\colon (G_\alpha) \to (H_\beta)$ is a \emph{pro-isomorphism} if it represents an isomorphism in the category of pro-groups, in other words, if
there exists a morphism of inverse systems $g\colon (H_\beta)\to(G_\alpha)$ such that $fg$ and $gf$ are each equivalent to identity morphisms.
An inverse system is \emph{stable} if it is pro-isomorphic to a rudimentary inverse system.
\end{defn}

If $\mathcal{A}'$ is a cofinal subset of the directed set $\mathcal{A}$, then the restriction of an inverse system $(G_\alpha)_{\alpha\in\mathcal{A}}$ to the subsystem indexed by $\mathcal{A}'$ is a pro-isomorphism (see \cite{MardesicSegal82}, \S I.1.1, Thm.~1).

The \emph{inverse limit} of an inverse system of groups $(G_\alpha)$ is the group
\[
   \varprojlim (G_\alpha)
   =
   \Bigl\{\, (g_\alpha) \in \prod_{\alpha\in \mathcal{A}} G_\alpha \Bigm| \text{$p_\alpha^{\alpha'}(g_{\alpha'}) = g_{\alpha}$ for all $\alpha\le\alpha'$} \,\Bigr\}.
\]

The following well-known result is a basic tool for the study of inverse limits.  (See, for example, \cite[\S I.5.1]{MardesicSegal82}.)

\begin{prop}
\label{prop:Morphisms}
A morphism of inverse systems of groups induces a unique homomorphism of their inverse limits via the universal mapping property of inverse limits.
Equivalent morphisms induce the same homomorphism. 
The assignment $(G_\alpha) \longmapsto \varprojlim (G_\alpha)$ is a covariant functor.
A pro-isomorphism of inverse systems induces an isomorphism of inverse limits.
\end{prop}

\begin{defn}
An \emph{exact sequence of inverse systems} with a common index set $\mathcal{A}$ consists of a sequence of level morphisms
\[
   \cdots \longrightarrow (A_\alpha) \longrightarrow (B_\alpha) \longrightarrow (C_\alpha) \longrightarrow \cdots
\]
such that for each fixed $\alpha\in\mathcal{A}$ the sequence of terms
\[
   \cdots \longrightarrow A_\alpha \longrightarrow B_\alpha \longrightarrow C_\alpha \longrightarrow \cdots
\]
is exact.
\end{defn}

Dydak \cite{dydak} proved the following analogue of the five lemma for an exact sequence of inverse sequences. 

\begin{lem}[Five lemma for stability]
\label{lem:FiveLemma}
Given an exact sequence of inverse sequences \textup{(}indexed by natural numbers\textup{)}
\[
   (A_m) \longrightarrow (B_m) \longrightarrow (C_m) \longrightarrow (D_m) \longrightarrow (E_m)
\]
in which $(A_m)$, $(B_m)$, $(D_m)$, and $(E_m)$ are stable, then $(C_m)$ is stable.
\end{lem}

\begin{ques}
Does the conclusion of Lemma~\ref{lem:FiveLemma} hold in general for inverse systems over arbitrary index sets?
\end{ques}

For the purposes of this article, the conclusion of Lemma~\ref{lem:FiveLemma} is sufficient.  Indeed most inverse systems arising in geometric topology are pro-isomorphic to inverse sequences.  However the authors do not know of any obstruction to a five lemma for stability in the general category of pro-groups.  We note that the methods of Dydak's proof do not immediately extend to the general case.

Exactness of inverse systems is not always preserved when passing to the inverse limit
\cite{EilenbergSteenrod_Axioms,Milnor_Steenrod}. However inverse limits are better behaved if all inverse systems involved are stable, as illustrated by the following result.

\begin{lem}[Exactness of stable inverse limits]
\label{lem:lim1}
Given a long exact sequence of stable inverse systems of groups
\[
   \cdots \longrightarrow (A_\alpha) \longrightarrow (B_\alpha) \longrightarrow (C_\alpha) \longrightarrow \cdots
\]
the induced sequence of inverse limits
\[
   \cdots \longrightarrow \varprojlim(A_\alpha) \longrightarrow \varprojlim(B_\alpha) \longrightarrow \varprojlim(C_\alpha) \longrightarrow \cdots
\]
is also exact.
\end{lem}

\begin{proof}
Consider the induced commutative diagram in the category of pro-groups
\[
\xymatrix{
   \cdots \ar[r] & \varprojlim(A_\alpha) \ar[r] \ar[d] & \varprojlim(B_\alpha) \ar[r] \ar[d] & \varprojlim(C_\alpha) \ar[r] \ar[d] & \cdots   \\
   \cdots \ar[r] & (A_\alpha) \ar[r] & (B_\alpha) \ar[r] & (C_\alpha) \ar[r] & \cdots
}
\]
in which the groups in the first row are considered to be rudimentary pro-groups, and the vertical maps are canonical projections.
Since the inverse systems in the second row are stable, the vertical projections are isomorphisms of pro-groups (\cite{MardesicSegal82}, \S I.5.2, Thm.~2).
Therefore the first row of rudimentary pro-groups is an exact sequence. (See \cite[\S II.2.3]{MardesicSegal82} for more information about exactness in the nonabelian category of pro-groups.)
Note that an exact sequence of rudimentary pro-groups is also exact when considered as a sequence of groups, completing the proof.
\end{proof}

In the setting of inverse sequences (\emph{i.e.}, systems indexed by the natural numbers) exactness is also preserved under the more general hypothesis that the first derived limit $\lim^1$ vanishes for every term in the given exact sequence \cite{Milnor_Steenrod,dydak,geoghegan}.  We will not need this more general result here.

\section{Coarse geometry}
\label{sec:CoarseGeometry}

In this section we review concepts in coarse geometry.  We refer the reader to \cite{Roe_CoarseGeometry} for more details and references.

\begin{defn}
A map $f\colon Y \to X$ between metric spaces is a \emph{coarse embedding} if there exist unbounded increasing functions $\rho_-,\rho_+ \colon \R_{\ge 0} \to \R_{\ge 0}$ such that
\[
   \rho_- \bigl( d(y_1,y_2) \bigr)
   \le d \bigl( f(y_1),f(y_2) \bigr)
   \le \rho_+ \bigl( d(y_1,y_2) \bigr)
   \qquad \text{for all $y_1,y_2 \in Y$.}
\]
The functions $\rho_-$ and $\rho_+$ are the \emph{lower control} and \emph{upper control} functions.
We note that if $Y$ is a length space, then one can always choose the upper control to be an affine function; \emph{i.e.}, every coarse embedding of a length space is coarse Lipschitz (see \cite[Lemma~1.10]{Roe_CoarseGeometry}).
\end{defn}

\begin{defn}
A \emph{metric simplicial complex} is a connected, locally finite simplicial complex in which each simplex is isometric to the standard Euclidean simplex with edges of length one. A metric simplicial complex $X$ has \emph{bounded geometry} if the vertex links of $X$ have a uniformly bounded number of simplicies. If $K \subset X$ is a subcomplex and $R$ is a positive integer, then $N_R(K)$ is the combinatorial $R$--tubular neighborhood of $K$ obtained by taking the $R$--fold iterated closed star of $K$. 
\end{defn}

\begin{exmp}
Suppose a finitely generated group $G$ acts properly by simplicial automorphisms on a bounded geometry metric simplicial complex $X$.  Endow $G$ with the word metric for any finite generating set.  Then the orbit map $G \to X$ given by $g \mapsto g(x_0)$ is a coarse embedding for any choice of basepoint $x_0 \in X$.
\end{exmp}

Throughout this article all chain complexes and homology groups are considered in the simplicial category and all coefficients are integral.  
A chain complex $\mathcal{C}=\{C_i,\boundary_i\}$ may be considered as a single abelian group equipped with a grading $\mathcal{C}= \bigoplus C_i$.
The notations $C_{\circ}(X) = \bigl\{ C_i(X),\boundary_i\bigr\}$ and $C_c^\circ(X) = \bigl\{ C_c^i(X),\delta^i\}$ denote respectively the chain complex of simplicial chains and the cochain complex of compactly supported simplicial cochains on $X$, each with integer coefficients.

\begin{defn}
If $X$ is a simplicial complex and $A$ is any subset of $C_\circ(X)$, then the \emph{support} of $A$ is the smallest subcomplex $K$ of $X$ so that $A \subset C_\circ(K)$.
A homomorphism $h\colon C_i(X) \to C_c^j(X)$ has \emph{displacement} at most $D$ for some positive integer $D$ if for every simplex $\sigma$ of $X$, the cochain $h(\sigma)$ is supported in $N_D(\sigma)$.
A homomorphism $C_c^j(X) \to C_i(X)$ has \emph{displacement} at most $D$ if for every simplex $\sigma$ of $X$, the image of the composition $C_c^j(\bar\sigma) \to C_c^j(X) \to C_i(X)$ has support lying in $N_D(\sigma)$, where the map $C_c^j(\bar\sigma) \to C_c^j(X)$ is given by extending a cochain on the subcomplex $\bar\sigma$ to a cochain on $X$ by setting it to evaluate to zero on simplices of $X$ outside of $\bar\sigma$.
\end{defn}

\begin{defn}
\label{defn:acyclic}
An $n$--dimensional metric simplicial complex $X$ is \emph{uniformly $k$--acyclic} for some $k \leq n$ if for every $R_1 >0$ there exists $R_2 >0$ so that for each subcomplex $K \subset X$ of diameter at most $R_1$ the inclusion $K \rightarrow N_{R_2}(K)$ induces zero on $\tilde{H}_k$.
The complex $X$ is \emph{uniformly acyclic} if $X$ is uniformly $k$--acyclic for all $k \leq n$.

By the Mayer--Vietoris Theorem, if $X$ is a union of two subcomplexes $A$ and $B$ such that $A$, $B$, and $A\cap B$ are uniformly acyclic, then so is $X$.
\end{defn}

\begin{defn}
Let $Y$ and $X$ be finite dimensional metric simplicial complexes. A chain map $h\colon C_\circ(Y) \rightarrow C_\circ(X)$ or a chain homotopy $h\colon C_\circ (Y) \to C_{\circ + 1}(X)$ is
\emph{$D$--controlled} for a constant $D<\infty$ if for each simplex $\sigma$ of $Y$, the support of $h\bigl( C_\circ(\sigma) \bigr)$ has diameter at most $D$.
A chain map $h$ is a \emph{coarse embedding} if $h$ is $D$--controlled for some $D$ and there exists a proper function $\phi\colon\R_+ \rightarrow \R_+$ such that for each subcomplex $K \subset Y$ of diameter at least $r$, the support of $h\bigl( C_\circ(K)\bigr)$ has diameter at least $\phi(r)$. 
Similar notions for cochain maps and cochain homotopies are defined analogously.
\end{defn}

\begin{rem}
\label{rem:CoarseEmbeddingSubcomplex}
Suppose $Z'$ is a subcomplex of $Z$ such that the inclusion $Z' \hookrightarrow Z$ is a coarse embedding.
It follows immediately from the definitions that the induced chain map of simplicial chain complexes $C_\circ(Z') \to C_\circ(Z)$ is a coarse embedding.
\end{rem}

Coarse embeddings of spaces do not need to be continuous.  However, if the spaces are finite dimensional uniformly acyclic metric simplicial complexes, then a coarse embedding induces a chain map that is also a coarse embedding in the following sense.

\begin{defn}
\label{def:Approximate}
Given a coarse embedding of metric simplicial complexes  $f\colon Z \to X$, a chain map $f_\# \colon C_\circ(Z) \to C_\circ(X)$ of simplicial chain complexes \emph{$M$--approximates} $f$ for some number $M$ if for each simplex $\sigma^i$ of $Z$, the support of the chain $f_\#(\sigma^i)$ is within Hausdorff distance $M$ of the image $f(\sigma^i)$.
We say that $f_\#$ \emph{approximates} $f$ if it $M$--approximates for some $M$.
Note that an approximating chain map $f_\#$ of a coarse embedding is again a coarse embedding.
\end{defn}

\begin{lem}
\label{lem:ChainCoarseEmbedding}
Assume $Z$ and $X$ are metric simplicial complexes such that $X$ is uniformly acyclic and $Z$ is finite dimensional. For each coarse embedding $f\colon Z\to X$ there exists a chain map $f_\# \colon C_\circ(Z) \to C_\circ(X)$ approximating $f$.

Although $f_\#$ is not uniquely determined by $f$, it is well-defined in a weaker sense:  Let $f\colon Z\to X$ be a coarse embedding, and let $f_\#$ and $f'_\#$ be two chain maps that $M'$--approximate $f$. There exists a constant $D=D(Z,X,M',f)$ such that $f_\#$ and $f'_\#$ are chain homotopic by a $D$--controlled chain homotopy.
\end{lem}

\begin{proof}
We construct $f_\#$ on $C_\circ(Z) = \bigoplus_{i\ge 0} C_i(Z)$ by induction on the dimension $i$.
In dimension $0$, define $f_\#$ by arbitrarily choosing for each $0$--simplex $\sigma^0$ of $Z$ a $0$--simplex of $X$ within a distance $1$ of the point $f(\sigma^0)$.
In dimension $1$, define $f_\#$ using the upper control of $f$ (the boundary of a $1$--simplex has image in $X$ with a uniformly bounded diameter) and the uniform acyclicity of $X$ (a pair of points separated by a uniformly bounded distance is the boundary of a $1$--chain with uniformly bounded diameter).
In dimension $k>1$, define $f_\#$ inductively, since the boundary of a $k$--simplex is a uniformly bounded diameter ${(k-1)}$--cycle in $X$ by hypothesis, so it is the boundary of a uniformly bounded diameter $k$--chain by uniformly acyclic.
The uniform bounds obtained in this inductive argument do depend on the dimension $k$, but by hypothesis $Z$ is finite dimensional, so there exists an upper bound independent of dimension.

To see that the chain map obtained is an approximation in the desired sense, first note that each simplex $\sigma^k$ of $Z$ has diameter $1$, so its image $f(\sigma^k)$ has uniformly bounded diameter in $X$.  In particular, the image $f(\sigma^k)$ lies in a uniformly bounded neighborhood of the image of the $0$--skeleton of $\sigma^k$, which lies uniformly close to the support of the chain $f_\#(\sigma^k)$.

The construction of the controlled chain homotopy involves a similar induction on dimension.  The base of the induction relies on the fact that any two approximating chain maps are uniformly close when restricted to the $0$--skeleton because they are both uniformly close to $f$.
The induction follows as before using that $X$ is uniformly acyclic and $Z$ is finite dimensional.
\end{proof}

The previous lemma implies that coarse embeddings of metric simplicial complexes (as above) induce well-defined maps on homology and on cohomology with compact supports as explained in the following corollary. 

\begin{cor}
\label{cor:InducedByApproximating}
Let $f\colon Z \to X$ be a coarse embedding of metric simplicial complexes, where $Z$ is finite dimensional and $X$ is uniformly acyclic.  Any two chain maps approximating $f$ induce the same homomorphism in homology and also in cohomology with compact supports \textup{(}with integer coefficients\textup{)}.
\end{cor}

\begin{proof}
The claim regarding homology is immediate since two chain homotopic maps of chain complexes induce the same map on homology.

Simplicial cohomology with compact supports makes sense only when the spaces involved are locally finite.  Indeed local finiteness implies that the dual of the boundary map takes compactly supported cochains to compactly supported cochains.
Because an approximating chain map $f_\# \colon C_\circ(Z) \to C_\circ(X)$ is a coarse embedding, its dual $f^\#$ also takes compactly supported cochains to compactly supported cochains.  Since the dual of a chain map is a cochain map, each approximating chain map $f_\#$ induces a homomorphism on cohomology with compact supports.
Since the dual of a chain homotopy is a cochain homotopy, any two approximating chain maps $f_\#$ and $g_\#$ induce the same maps on cohomology with compact supports.
\end{proof}

\section{\v{C}ech homology of filtered ends}
\label{sec:DeepHomology}
The study of ends of a space $X$ largely involves properties of the inverse system of subspaces $X-K$ indexed by the family of bounded subsets $K \subset X$.  More generally, if $K$ is an unbounded subspace, one can study the analogous filtration of $X$ by metric neighborhoods of $K$.
For each subcomplex $K$ of a metric simplicial complex $X$, the family of simplicial neighborhoods $N_R(K)$ of $K$ is a filtration of $X$ in the sense of Geoghegan \cite[Chap.~14]{geoghegan}.  This inverse sequence is unchanged up to pro-isomorphism if one replaces $K$ with a subset $K'$ at finite Hausdorff distance from $K$.  In order to see this equivalence it is more convenient to work with a larger inverse system, as described below.
We note that the following definition may be considered as a special case of the general notion of homology of filtered ends of spaces discussed by Geoghegan.

\begin{defn}[\v{C}ech homology of filtered ends]
\label{defn:CechHomology}
Let $X$ be a metric simplicial complex.  For each subset $K$ of $X$, a \emph{neighborhood $L$ of $K$} is a subcomplex of $X$ such that $L$ is contained in a metric tubular neighborhood $N_R(K)$ of $K$ for some $R<\infty$.
The family $\mathcal{N}(K)$ of all neighborhoods of $K$ is directed by inclusion: $L\le L'$ if $L \subseteq L'$.
The $i$th \emph{homology pro-group of the $K$--end} is a pro-group associated to the end of the filtration $\mathcal{N}(K)$ and is defined by the inverse system
\[
   \text{pro-}H_i^D(K^c) = \bigl( H_i (\overline{X-L}) \bigr)
\]
indexed by $L \in \mathcal{N}(K)$.
Taking inverse limits we get the $i$th \emph{\v{C}ech homology of the $K$--end},
\[
   \check{H}_i^D(K^c) = \varprojlim\limits_L \bigl( H_i (\overline{X-L}) \bigr).
\]

Notice that by definition $\text{pro-}H_i^D(K_1^c) = \text{pro-}H_i^D(K_2^c)$ for any subsets $K_1$ and $K_2$ at a finite Hausdorff distance from each other, since $\mathcal{N}(K_1) = \mathcal{N}(K_2)$.  In the special case that $K$ is a bounded subset, the $i$th homology pro-group of the $K$--end is also known as the $i$th \emph{homology pro-group at infinity} of $X$ and does not depend on the choice of bounded set $K$.  Its inverse limit is the $i$th \emph{\v{C}ech homology group at infinity} of $X$.

If $K$ is the empty set, then by definition $\mathcal{N}(K)$ is also empty. In this degenerate case, we note that $\text{pro-}H_i^D(\emptyset^c)$ reduces to $H_i(X)$.
\end{defn}

Under certain conditions, a map of pairs $(A,B) \to (X,Y)$ induces a corresponding map $\check{H}_i^D(B^c) \to \check{H}_i^D(Y^c)$ on \v{C}ech homology of filtered ends.
Rather than developing the most general case, we focus on the following simple setup.
Consider a coarse embdedding of metric simplicial complexes $f\colon A \rightarrow X$, and let $B$ be a subcomplex of $A$ with image $f(B)$.
Roughly speaking, any such $f$ induces a homomorphism from \v{C}ech homology of the $B$--end of $A$ to the \v{C}ech homology of the $f(B)$--end of $X$, provided that we first choose a chain map that approximates $f$.

\begin{prop}
\label{ap}
Let $f\colon A \rightarrow X$ be a coarse embedding of metric simplicial complexes, where $A$ is finite dimensional and $X$ is uniformly acyclic. Let $B$ be any subcomplex of $A$.
Then there exists a unique homomorphism 
\[
   \check{H}_i^D(B^c) \xrightarrow[\qquad]{\check{f}} \check{H}_i^D\bigl( (fB)^c \bigr)
\]
induced by $f$ in the following sense.  Every chain map
\[
   C_\circ(A) \xrightarrow[\qquad]{f_\#} C_\circ(X)
\]
approximating $f$ induces \textup{(}by restriction to subspaces\textup{)} a morphism between the inverse systems defining the \v{C}ech homology groups.  This morphism induces a homomorphism $\check{f}$ of inverse limits that does not depend on any of the choices made in its construction.
\end{prop}

\begin{proof}
We first show that any approximating chain map $f_\#$ induces a map on deep homology via restrictions.
Suppose $f_\#$ is a chain map approximating $f$, and let $K$ be the support of $f_\#\bigl( C_\circ(B) \bigr)$.
By construction, the Hausdorff distance between $K$ and $f(B)$ is finite, for each simplex $\sigma^k$ of $A$ the support of the chain $f_\#(\sigma^k)$ is within a uniformly bounded distance of the image $f(\sigma^k)$, and $f$ is a coarse embedding.
Therefore for each $R$ there is a constant $M$ such that whenever $L \subseteq N_R\bigl( f(B) \bigr)$ we have
\begin{equation}
\label{eqn:RelativeEnds}
   f_\#\bigl( C_\circ(\overline{A-N_{M}(B)}) \bigr) \subseteq C_\circ \bigl( \overline{X-N_{R}(f(B))} \bigr)
   \subseteq C_\circ \bigl( \overline{X-L} \bigr).
\end{equation}
In particular $f_\#$ induces, via restriction, a homomorphism
\[
   H_i \bigl( \overline{A - N_{M}(B)} \bigr)
   \xrightarrow[\qquad]{f_L}
   H_i \bigl( \overline{X - L} \bigr).
\]
For concreteness, for each $L$ we choose $R=R(L)$ and $M=M(L)$ to be the minimal natural numbers satisfying \eqref{eqn:RelativeEnds}, so that the mapping $L\mapsto M(L)$ is nondecreasing.
Since at the level of chain maps the family of such restrictions commutes with inclusions, the maps $f_L$ define a morphism of inverse systems, which induces a corresponding homomorphism $\check{f}\colon \check{H}_i^D(B^c)\to \check{H}_i^D(K^c)$.

Now suppose $f_\#$ and $f'_\#$ are two chain maps that approximate $f$.  We will show that the induced maps $\check{f}$ and $\check{f}'$ constructed as above are equal.
Recall that by Lemma~\ref{lem:ChainCoarseEmbedding} the chain maps $f_\#$ and $f'_\#$ are chain homotopic by a $D$--controlled chain homotopy $F$ for some $D$ depending on $f_\#$ and $f'_\#$.
Since $f_\#$ is a coarse embedding, the chain homotopy $F$ is also a coarse embedding.
Suppose $L \in \mathcal{N}\bigl(f(B)\bigr)$, and choose $M'$ sufficiently large that the chain homotopy $F$ restricts to a map
\[
   C_\circ (\overline{A-N_{M'}(B)}) \longrightarrow
   C_{\circ+1} \bigl( \overline{X-L} \bigr).
\]
Then the two homology maps
\[
   H_i (\overline{A-N_{M'}(B)}) \longrightarrow
   H_i \bigl( \overline{X-L} \bigr)
\]
induced by restriction of $f_\#$ and $f'_\#$ are equal.  It follows that the induced maps
$\check{H}_i^D(B^c)\to \check{H}_i^D(K^c)$ are equal as well.
\end{proof}

\begin{rem}
Throughout this section we have discussed the ordinary \v{C}ech homology of filtered ends.  The analogous reduced \v{C}ech homology groups are defined similarly.  We omit the details, which are essentially the same as in the ordinary case.
\end{rem}

\subsection{Relative homology of filtered ends} \label{sec:rel_homology}

This section discusses a relative version of \v{C}ech homology of filtered ends.
In general \v{C}ech homology theories do not satisfy the exactness axiom of Eilenberg--Steenrod.
However in certain well-behaved situations exactness still holds.  One such natural situation is when $X$ is a coarse Poincar\'{e} duality space, as discussed in the following definition.

\begin{defn}
\label{def:coarsePDn}
A \emph{coarse $PD(n)$ space} $X$ is a bounded geometry, uniformly acyclic, metric simplicial complex equipped with cochain maps
\[
   C_{n-\circ}(X) \xrightarrow[\qquad]{\bar{P}} C_c^{\circ}(X) \xrightarrow[\qquad]{P} C_{n-\circ}(X)
\]
such that for some integer constant $E_0$ we have
\begin{enumerate}
    \item $P$ and $\bar{P}$ have displacement at most $E_0$.
    \item $\bar{P} \of P$ and $P\of \bar{P}$ are cochain homotopic to the identity by $E_0$--controlled cochain homotopies.
\end{enumerate}
\end{defn}

\begin{defn}[Relative \v{C}ech homology]
Let $X$ be a metric simplicial complex and let $K_1\subset K_0$ be subsets of $X$.
A \emph{neighborhood} of the pair $(K_0,K_1)$ is a pair of subcomplexes $(L_0,L_1)$ of $X$ with $L_1 \subseteq L_0$ such that there exists $R$ with $L_i \subseteq N_R(K_i)$.
As in the absolute case, the family $\mathcal{N}(K_0,K_1)$ of all neighborhoods of $(K_0,K_1)$ is directed by inclusion, and we define the $i$th \emph{homology pro-group of the $(K_0,K_1)$--end}, 
\[
   \text{pro-}H_i^D(K_1^c,K_0^c) = \bigl( H_i (\overline{X-L_1}, \overline{X-L_0}) \bigr)
\]
indexed by $(L_0,L_1) \in \mathcal{N}(K_0,K_1)$.
Taking inverse limits we get the corresponding $i$th \emph{\v{C}ech homology group of the $(K_0,K_1)$--end},
\[
   \check{H}_i^D(K_1^c,K_0^c) = \varprojlim\limits \bigl( H_i (\overline{X-L_1}, \overline{X-L_0}) \bigr).
\]

Suppose $X$ is held fixed.  An inclusion $(C,D) \to (A,B)$ of pairs of subsets induces a morphism of homology pro-groups:
\[
   \text{pro-}H_i^D(B^c,A^c)
   \longrightarrow
   \text{pro-}H_i^D(D^c,C^c).
\]
Thus the relative homology pro-group of filtered ends may be considered as a contravariant functor on the pairs of subsets of $X$.
\end{defn}

When $X$ is a coarse $PD(n)$ space, one may identify the homology pro-group associated to the pair $(K,\emptyset)$ with the reduced homology pro-group associated to $K$ as defined in Definition~\ref{defn:CechHomology}.
Indeed a neighborhood of $(K,\emptyset)$ has the form $(L,\emptyset)$ for some neighborhood $L$ of $K$.
Since $X$ is acyclic, the connecting homomorphism $\boundary$ of the reduced homology long exact sequence of the pair $(X,X-L)$ is an isomorphism in every dimension:
\[
   H_i({X}, \overline{X-L})
   \xrightarrow[\qquad]{\boundary}
   \widetilde{H}_{i-1} (\overline{X-L})
\]

    \subsection{Stable deep components} \label{sec_stable_deep_comps}

We describe what the $0$th \v{C}ech homology of the $K$--end represents in terms of a natural basis that we use later. 

\begin{defn}
Let $X$ be a metric simplicial complex, and let $K$ be a subset of $X$.
The set of \emph{ends of the pair} $(X,K)$ or \emph{relative ends} is the set
\[
   \mathcal{E}(X,K)
   = \varprojlim \pi_0 \bigl( \overline{X - L} \bigr).
\]
indexed by the directed system of all subcomplexes $L \in \mathcal{N}(K)$.
\end{defn}

A component $U \in \pi_0 \bigl( \overline{X - L} \bigr)$ is {\it deep} if it is in the image of the natural projection $\mathcal{E}(X,K) \to \pi_0 \bigl( \overline{X - L} \bigr)$, or equivalently if $U$ is not contained in a bounded neighborhood of $K$.
A deep component of $\overline{X - L}$ is \emph{stable} if it intersects exactly one deep component of $\overline{X - L'}$ for each $L'\in \mathcal{N}(K)$ with $L\subseteq L'$.

Note that if $U$ is a stable deep component of $\overline{X - L}$, then there exists a unique end of the pair $(X,K)$ mapping to $U$ under the natural projection.

The relation between stable deep components and the $0$th \v{C}ech homology of the $K$--end is summarized in the following result.

\begin{prop}[\cite{geoghegan}, \S14.3] \label{prop_H0_gens}
Suppose the inverse system $\{ H_0(\overline{X - L})\}_{L \in \mathcal{N}(K)}$ is stable and its inverse limit is free abelian of finite rank.
Then the set of relative ends $\mathcal{E}(X,K)$ is finite and freely generates $\check{H}_0^D(K^c)$.
Furthermore there exists $R_0<\infty$ such that whenever $L \in \mathcal{N}(K)$ satisfies $N_{R_0}(K) \subseteq L$ all deep components $U \in \pi_0 \bigl( \overline{X - L} \bigr)$ are stable, and the natural projection
\[
   \mathcal{E}(X,K) \to \pi_0\bigl( \overline{X - L} \bigr)
\]
is a bijection onto the set of deep components of $\overline{X - L}$.

In addition,  for each $L \in \mathcal{N}(K)$ there exists $L' \in \mathcal{N}(K)$ such that any component $U \in \pi_0\bigl( \overline{X - L} \bigr)$ that is not deep is contained in $L'$.
\end{prop}

\section{Coarse Alexander duality}
\label{sec:CoarseAlexanderDuality}

In this section, we review the coarse Alexander duality isomorphism theorem of Kapovich--Kleiner \cite{kapovichkleiner05}.

Suppose $X$ is a coarse $PD(n)$ space and $Z$ is a uniformly acyclic metric simplicial complex with bounded geometry. Let $f \colon Z\to X$ be a coarse embedding.
For each $k$, coarse Alexander duality is a pro-group morphism
\[
   \text{pro-}\widetilde{{H}}{}_{n-k-1}^D(fZ^c) \xrightarrow[\qquad]{A}
   H_c^k(Z),
\]
defined as follows.

Although by definition the homology pro-group is indexed by the family of subcomplexes $\mathcal{N}(fZ)$, it suffices to restrict our attention to the cofinal sequence of subcomplexes $\bigset{N_R(K)}{R=1,2,3,\dots}$, where $K$ is any fixed subcomplex at a finite Hausdorff distance from the set $fZ$.  Applying this restriction, we may consider the homology pro-group to be an inverse sequence indexed by the natural number $R$.

Fix a subcomplex $K \in \mathcal{N}(fZ)$ at a finite Hausdorff distance from $fZ$, and let
\[
   Y^R = \overline{X - N_R(K)}.
\]
Because the cochain map $\bar{P}$ given by Definition~\ref{def:coarsePDn} has bounded displacement, it restricts to a cochain map
\[
   C_{n-\circ}(Y^{R+E})
   \xrightarrow[\qquad]{\bar{P}}
   C_c^\circ \bigl( X,N_R (K) \bigr),
\]
where $E=E_0+1$ and $E_0$ is the constant from Definition~\ref{def:coarsePDn}.
Indeed $\bar{P}$ maps each chain in $Y^{R+E}$ to a cochain supported in $N_{E_0}(Y^{R+E}) = N_{E_0}(Y^{R+1+E_0})$, which lies in $\overline{X- N_{R+1}(K)}$.  Such a cochain evaluates to zero on every simplex of $N_{R}(K)$, so it lies in the given relative cochain complex.
Therefore $\bar{P}$ induces a cochain map, also denoted by $\bar{P}$, between the two quotient cochain complexes in the following commutative diagram of cochain maps:
\[
\xymatrix{
   0 \ar[r] & C_{n-\circ}(Y^{R+E}) \ar[r] \ar[d]^{\bar{P}} & C_{n-\circ}(X) \ar[r] \ar[d]^{\bar{P}} & C_{n-\circ}(X,Y^{R+E}) \ar[r] \ar[d]^{\bar{P}} & 0 \\
   0 \ar[r] & C_c^\circ \bigl(X,N_R(K) \bigr) \ar[r] & C_c^\circ (X) \ar[r] & C_c^\circ \bigl( N_R(K) \bigr) \ar[r] & 0
}
\]

Choose an approximating chain map $f_\# \colon C_\circ(Z) \to C_\circ(X)$.  
We now add the assumption that $K$ has been chosen large enough so as to contain the support of $f_\#\bigl( C_\circ(Z) \bigr)$.
This assumption is sufficient for defining a morphism of pro-groups, since the collection of sets $\bigl\{N_R(K)\bigr\}$ is cofinal in $\mathcal{N}(fZ)$.
Thus we may replace $K$ with its $R_0$--neighborhood for a large enough value of $R_0$.  In particular, the chain map $f_\#$ restricts to a chain map
\[
   C_\circ(Z) \xrightarrow[\qquad]{f_\#} C_\circ \bigl( N_R(K) \bigr).
\]
As discussed in Corollary~\ref{cor:InducedByApproximating}, a chain map that is a coarse embedding induces a dual cochain map of compactly supported cochain complexes.  In the present setting, we get an induced cochain map
\[
   C_c^\circ \bigl( N_R(K) \bigr)
   \xrightarrow[\qquad]{f^\#}
   C_c^\circ (Z).
\]

The connecting homomorphism $\boundary$ of the reduced homology long exact sequence of the pair $(X,Y^{R+E})$ is an isomorphism in every dimension because $X$ is acyclic.
Consider, for each fixed $k$, the map
$A^{R+E} \colon \widetilde{H}_{n-k-1}(Y^{R+E}) \to H_c^k (Z)$ defined by the following composition of homomorphisms:
\[
\xymatrix{
   \widetilde{H}_{n-k-1}(Y^{R+E})
   \ar[r]^{\boundary^{-1}} &
   H_{n-k} (X,Y^{R+E})
   \ar[r]^{\ \ \bar{P}} &
   H_c^k \bigl( N_R(K) \bigr)
   \ar[r]^{\ \ \ \ f^*} &
   H_c^k (Z)
}
\]
in which $\boundary^{-1}$ denotes the inverse of the connecting homomorphism $\boundary$, and $\bar{P}$ and $f^*$ are the induced maps on homology/cohomology discussed above.
When $R<R'$ consider the following diagram in which the vertical maps are induced by inclusion:
\[
\xymatrix{
   \widetilde{H}_{n-k-1}(Y^{R'+E})
   \ar[r]^{\boundary^{-1}} \ar[d] &
   H_{n-k} (X,Y^{R'+E})
   \ar[r]^{\ \ \bar{P}} \ar[d] &
   H_c^k \bigl( N_{R'}(K) \bigr)
   \ar[r]^{\ \ \ \ f^*} \ar[d] &
   H_c^k (Z) \ar[d]^{id} \\
   \widetilde{H}_{n-k-1}(Y^{R+E})
   \ar[r]^{\boundary^{-1}} &
   H_{n-k} (X,Y^{R+E})
   \ar[r]^{\ \ \bar{P}} &
   H_c^k \bigl( N_R(K) \bigr)
   \ar[r]^{\ \ \ \ f^*} &
   H_c^k (Z)
}
\]
Observe that the diagram above commutes.  Indeed, the leftmost square commutes by naturality of the connecting homomorphism $\boundary$.  The middle square commutes at the level of cochain complexes, since the two horizontal maps labelled $\bar{P}$ are induced by the same map $\bar{P}\colon C_{n-\circ}(X)\to C_c^\circ(X)$.  Similarly the righthand square commutes at the level of cochains by the definition of $f^\#$.

Recall that a morphism of inverse systems with rudimentary target is determined by a single homomorphism of groups.
In particular, for each $R>E$ the map $A^R$ determines a morphism
\[
   \bigl(\widetilde{H}_{n-k-1}(Y^{R}) \bigr)_R
   \xrightarrow[\qquad]{A^{R}}
   \bigl( H_c^k(Z) \bigr)
\]
of inverse systems.
The argument in the previous paragraph implies that the morphisms $A^R$ and $A^{R'}$ are equivalent for all $R,R'>E$.
The equivalence class $A = [A^R]$ is a morphism of pro-groups, known as \emph{coarse Alexander duality}.

Kapovich--Kleiner have shown the following coarse Alexander duality isomorphism theorem.

\begin{thm}[\cite{kapovichkleiner05}, Theorem~7.7]
\label{thm:CAD}
Suppose $Z$ is a uniformly acyclic metric simplicial complex with bounded geometry, and suppose $X$ is a coarse $PD(n)$ space.
Let $f \colon Z\to X$ be a coarse embedding.
Then the coarse Alexander duality morphism
\[
   \textup{pro-}\widetilde{H}^D_{n-k-1} ( fZ^c )_R \xrightarrow[\qquad]{A}
    H_c^k(Z)
\]
is a pro-isomorphism.
In particular, the reduced homology pro-groups of the $(fZ)$--end of $X$ are stable, and $A$ induces an isomorphism
\[
   \check{\widetilde{H}}{}_{n-k-1}^D(fZ^c) \xrightarrow[\qquad]{A}
   H_c^k(Z)
\]
called the coarse Alexander duality isomorphism.
\end{thm}

\section{Coarse Alexander duality for pairs}
\label{sec:DualityForPairs}

The ordinary coarse Alexander duality isomorphism theorem deals with a single coarse embedding $Z \to X$.
In this section, we establish coarse Alexander duality for pairs.  More specifically we introduce a duality theorem for a coarse embedding of a pair of complexes $Z_0\subseteq Z_1$ in a coarse $PD(n)$ space $X$.
In a general context, \v{C}ech homology theories do not satisfy the Eilenberg--Steenrod axioms since they do not satisfy the exactness axiom for pairs.  However in the setting of coarse $PD(n)$ spaces, coarse Alexander duality implies that the inverse system defining the homology pro-group of the filtered end is stable.  Thus in this setting we establish the exactness of the long exact sequence of a pair.
The main objective of this section is to explain the details of this exactness statement.

More precisely the coarse Alexander duality isomorphism for pairs is proved using the long exact sequence of a pair in the setting of \v{C}ech homology of relative ends, which we show is exact and satisfies a naturality property with respect to the coarse Alexander duality isomorphism.

Suppose $X$ is a coarse $PD(n)$ space, and $Z_1 \subset Z_0$ is a pair of uniformly acyclic complexes with bounded geometry such that the inclusion $Z_1\hookrightarrow Z_0$ is a coarse embedding.  
If $f\colon Z_0 \to X$ is a coarse embedding, then it restricts to a coarse embedding $f\colon Z_1\to X$.
In this context we introduce a relative version of \v{C}ech homology associated to the complement of a pair of subcomplexes, defined as follows.
The coarse embedding $f$ induces an approximating chain map $f_\#\colon C_\circ(Z_0) \to C_\circ(X)$ by Lemma~\ref{lem:ChainCoarseEmbedding}.
We let $K_i$ denote the subcomplex of $X$ that is the support of $f_\#\bigl(C_\circ(Z_i)\bigr)$ for each $i$.  
The relative \v{C}ech homology of the $(Z_0,Z_1)$--end is given by
\[
   \check{H}^D_i(fZ_1^c, fZ_0^c) = \varprojlim\limits_R H_{i}(Y_1^{R},Y_0^{R})
\]
where $Y^R_i = \overline{X - N_R(K_i)}$.

As in the definition of coarse Alexander duality, the cochain maps
\[
   \bar{P} \colon C_{n-\circ}(X) \to C_c^\circ(X)
   \qquad \text{and} \qquad 
   f^\# \colon C_c^\circ(X) \to C_c^\circ(Z_0)
\]
induce cochain maps that make the following diagram commute:

\begin{equation*}
\xymatrix@C-12pt{
   0 \ar[r] &
   C_{n-\circ}(Y_1^{R+E},Y_0^{R+E}) \ar[r]\ar[d]^{\bar{P}} & 
   C_{n-\circ}(X,Y_0^{R+E}) \ar[r] \ar[d]^{\bar{P}}  &
   C_{n-\circ}(X,Y_1^{R+E})  \ar[r] \ar[d]^{\bar{P}}  &
   0 \\
   0 \ar[r] &
   C_c^\circ \bigl( N_{R}(K_0),N_{R}(K_1) \bigr) \ar[r] \ar[d]^{f^\#} &
   C_c^\circ\bigl( N_{R}(K_0) \bigr) \ar[r] \ar[d]^{f^\#} &
   C_c^\circ\bigl( N_{R}(K_1) \bigr) \ar[r] \ar[d]^{f^\#} &
   0 \\
   0 \ar[r] &
   C_c^\circ(Z_0,Z_1) \ar[r] &
   C_c^\circ(Z_0)   \ar[r] &
   C_c^\circ(Z_1)    \ar[r] &
   0
}
\end{equation*}

We define the \emph{coarse Alexander duality} morphism of pairs to be the equivalence class $A=[A^{R+E}]$ consisting of the homology maps induced by the leftmost column of the above diagram:
\[
   \bigl( {H}_{n-k}(Y_1^{R},Y_0^R) \bigr)_R
   \xrightarrow[\qquad]{A}
   H_c^k(Z_0,Z_1)
\]
The maps $A^R$ represent a well-defined morphism of pro-groups by similar reasoning to that used in Section~\ref{sec:CoarseAlexanderDuality} for the original coarse Alexander duality morphism.

The definitions of the coarse Alexander duality morphism for a single coarse embedding and for a pair of coarse embeddings are related by the following commutative diagram of homology groups:
\[
\xymatrix@C-5pt{
   \ar[rd] & & 
   \tilde{H}_{n-k-1}(Y_0^{R}) \ar[r]  &
   \tilde{H}_{n-k-1}(Y_1^{R})  \ar[rd]  & \\
   \cdots \ar[r]_{\!\!\!\!\!\!\!\!\!\!\!\!\!\!\!\!\!\!\!\!\!\boundary} &
   H_{n-k}(Y_1^{R},Y_0^{R}) \ar[r] \ar[d] \ar[ur]^{\boundary} & 
   H_{n-k}(X,Y_0^{R}) \ar[r] \ar[d] \ar[u]_{\boundary} &
   H_{n-k}(X,Y_1^{R})  \ar[r]_{\ \ \ \ \ \ \ \boundary} \ar[d] \ar[u]_{\boundary} &
   \cdots \\
   \cdots \ar[r]^{\!\!\!\!\!\!\!\!\!\!\!\!\!\!\!\delta} &
   H_c^k(Z_0,Z_1) \ar[r] &
   H_c^k(Z_0)   \ar[r] &
   H_c^k(Z_1)    \ar[r]^{\ \ \ \delta} &
   \cdots
}
\]
The lower half of this diagram of homology maps is induced by the preceding commutative diagram of cochain maps.
The wavy top row of the diagram consists of the reduced homology long exact sequence of the pair $(Y_1^R,Y_0^R)$.
Each map labeled $\boundary$ in the diagram is the connecting homomorphism of a homology long exact sequence.
One can easily verify that this homology diagram commutes.  Indeed the lower half commutes due to the well-known fact that a morphism between short exact sequences of chain complexes always induces a morphism between the corresponding homology long exact sequences. (See, for instance, \cite[Thm.~4.5.4]{Spanier_AlgebraicTop}.)
Similarly each region in the upper half of the diagram commutes by naturality of the connecting homomorphism $\boundary$ with respect to inclusion maps.

We now condense the above homology diagram, showing only the top and bottom rows. Observe that the vertical maps in the condensed diagram below are, by definition, equal to the maps in the coarse Alexander duality morphism.
\begin{equation}
\label{eqn:DualityofPairs}
\begin{gathered}
\xymatrix@C-10pt{
   \cdots \ar[r] &
   H_{n-k}(Y_1^{R},Y_0^{R}) \ar[r]^{\,\,\boundary} \ar[d]^{A_R} & 
   \tilde{H}_{n-k-1}(Y_0^{R}) \ar[r] \ar[d]^{A_R}  &
   \tilde{H}_{n-k-1}(Y_1^{R})  \ar[r] \ar[d]^{A_R}  &
   \cdots \\
   \cdots \ar[r]^{\!\!\!\!\!\!\!\!\!\!\!\!\!\!\!\delta} &
   H_c^k(Z_0,Z_1) \ar[r] &
   H_c^k(Z_0)   \ar[r] &
   H_c^k(Z_1)    \ar[r]^{\ \ \ \delta} &
   \cdots
}
\end{gathered}
\end{equation}

\begin{thm}[Coarse Alexander duality for pairs]
\label{thm:CAD_Pairs}
There exists a commutative diagram 
\[
\xymatrix@C-6pt{
   \cdots \ar[r] &
   \displaystyle
   \check{H}_{n-k}^D(fZ_1^c,fZ_0^{c}) \ar[r]^{\,\,\,\,\,\boundary} \ar[d]^{A}_{\cong} & 
   \displaystyle \check{\widetilde{H}}{}^D_{n-k-1}(fZ_0^{c}) \ar[r] \ar[d]^{A}_{\cong}  &
   \displaystyle \check{\widetilde{H}}{}^D_{n-k-1}(fZ_1^{c})  \ar[r] \ar[d]^{A}_{\cong}  &
   \cdots \\
   \cdots \ar[r]^{\!\!\!\!\!\!\!\!\!\!\!\!\!\delta} &
   H_c^k(Z_0,Z_1) \ar[r] &
   H_c^k(Z_0)   \ar[r] &
   H_c^k(Z_1)    \ar[r]^{\ \ \ \delta} &
   \cdots
}
\]
in which the first row is an exact sequence induced by the long exact sequence of the pair $(Y_1^R,Y_0^R)$, the second row is a long exact sequence of the pair $(Z_0,Z_1)$,
and the vertical maps are coarse Alexander duality isomorphisms.
\end{thm}

\begin{proof}
Passing to the inverse limit with respect to $R$ in \eqref{eqn:DualityofPairs} gives the desired commutative diagram.
By coarse Alexander duality, the second and third columns of \eqref{eqn:DualityofPairs} are pro-isomorphisms.  Hence the second and third vertical maps in the limiting diagram are isomorphisms.  If we knew that the limiting diagram had exact rows, we would be done by the traditional five lemma.
Since the second row is known to be an exact sequence, we only need to show that the top row of the limiting diagram is an exact sequence.

The inverse sequences in the second and third columns of \eqref{eqn:DualityofPairs} are stable since each is pro-isomorphic to a rudimentary inverse system.
By Lemma~\ref{lem:FiveLemma} (the five lemma for stability), the inverse sequence in the first column is stable as well.
Since the first row of \eqref{eqn:DualityofPairs} is an exact sequence of stable inverse systems, its inverse limit is again exact by Lemma~\ref{lem:lim1}.
\end{proof}

We now describe a setting in which inclusions of pairs induce homomorphisms between the corresponding \v{C}ech homology groups of relative ends.

\begin{defn}
A \emph{coarse inclusion} is an inclusion of simplicial complexes $B_0 \subseteq A_0$ such that the inclusion map is a coarse embedding.
A \emph{coarse inclusion of pairs} is an inclusion of simplicial pairs $(B_0,B_1) \to (A_0,A_1)$ such that each of the inclusions
\[
\xymatrix{
   B_1 \ar[r] \ar[d] & B_0 \ar[d] \\
   A_1 \ar[r]        & A_0
}
\]
is a coarse embedding.
\end{defn}

\begin{prop}
\label{prop:HomologyAxioms}
Suppose $(B_0,B_1) \to (A_0,A_1)$ is a coarse inclusion of pairs such that $A_0,A_1,B_0,B_1$ are uniformly acyclic with bounded geometry,
and suppose $f \colon A_0 \to X$ is a coarse embedding into a coarse $PD(n)$ space $X$.
\begin{enumerate}
   \item There exists a commutative diagram 
\[
\xymatrix{
   \check{H}_i^D (fA_1^c,fA_0^c) \ar[r] \ar[d]^A_{\cong} & \check{H}_i^D (fB_1^c,fB_0^c) \ar[d]^A_{\cong} \\
   H_c^{n-i}(A_0,A_1) \ar[r]        & H_c^{n-i}(B_0,B_1)
}
\]
in which the vertical maps are Alexander duality isomorphisms and the horizontal maps are induced by inclusions of spaces.
    \item \textup{(}Naturality\textup{)} The connecting homomorphism in the \v{C}ech homology long exact sequence is natural with respect to coarse inclusions, in the sense that the following diagram commutes:
\[
\xymatrix{
   \check{H}_i^D (fA_1^c,fA_0^c) \ar[r]^{\ \ \boundary} \ar[d] & \check{\widetilde{H}}{}^D_{i-1}(fA_0^{c}) \ar[d] \\
   \check{H}_i^D (fB_1^c,fB_0^c) \ar[r]^{\ \ \boundary}        & \check{\widetilde{H}}{}^D_{i-1}(fB_0^{c})
}
\]    
    \item \textup{(}Excision\textup{)}
    Suppose $(B_0,B_1) = (\overline{A_0 - U},\overline{A_1-U})$ for some subcomplex $U$ of $A_0$ contained in the interior of $A_1$. Then the map induced by inclusion
\[
   \check{H}_i^D (fA_1^c,fA_0^c)
   \longrightarrow
   \check{H}_i^D (fB_1^c,fB_0^c)
\]
is an isomorphism.
\end{enumerate}
\end{prop}

\begin{proof}
As in the definition of coarse Alexander duality for pairs, 
the cochain maps
\[
   \bar{P} \colon C_{n-\circ}(X) \to C_c^\circ(X)
   \qquad \text{and} \qquad 
   f^\# \colon C_c^\circ(X) \to C_c^\circ(A_0)
\]
induce morphisms that make the following diagram commute:
\begin{equation}
\label{eqn:Excision}
\begin{gathered}
\xymatrix{
   \text{pro-}H_i^D(fA_1^c,fA_0^c) \ar[d]^A \ar[r] & \text{pro-}H_i^D(fB_1^c,fB_0^c) \ar[d]^A \\
   H_c^{n-i}(A_0,A_1) \ar[r] & H_c^{n-i}(B_0,B_1)
}
\end{gathered}
\end{equation}
We obtain (1) by passing to the inverse limit in the above diagram.

The following diagram commutes
\[
\xymatrix{
   \text{pro-}H_{i}^D(f A_1^c,f A_0^c) \ar[r]^{\ \ \partial} \ar[d] &
      \text{pro-}\tilde{H}_{i-1}^D(f A_0^c) \ar[d]  \\
    \text{pro-}H_{i}^D(f B_1^c,f B_0^c) \ar[r]^{\ \ \partial}  & \text{pro-}\tilde{H}_{i-1}^D(f B_0^c)
}
\]
by naturality of the connecting homomorphism in simplicial homology.  Again, (2) follows by passing to the inverse limit.

Conclusion (3) follows from (1) as follows.
In \eqref{eqn:Excision}, the horizontal map in the second row is an isomorphism by excision for simplicial cohomology with compact supports.
We conclude that the horizontal map in the first row of \eqref{eqn:Excision} is a pro-isomorphism, since the other three maps in \eqref{eqn:Excision} are pro-isomorphisms.
\end{proof}

\section{The Jordan cycle and adjacency graph} \label{sec:Jordan_adj_graph}

Let $X$ be a coarse $PD(n)$ space.
In this section, we examine the structure of a coarse embedding of a union of $k$ halfspaces of dimension $n-1$ into $X$.  We show that such a union coarsely separates $X$ into $k$ deep components that are arranged in a cyclic pattern.  The main tools used to understand this structure are the Jordan cycle and an associated dual graph to the system of complementary deep components.  These objects are defined and studied below.

\subsection{A topological adjacency graph}

In order to motivate some of the machinery used in the coarse setting, we begin by examining a more classical situation in the topological category.

The classical Jordan--Brouwer separation theorem states that any subspace $T$ of $S^n$ homeomorphic to $S^{n-1}$ separates $S^n$ into exactly two complementary components, each of which has $T$ as its frontier (see, for instance, \cite[Thm.~4.7.15]{Spanier_AlgebraicTop}).  In this subsection, we extend this result to study embeddings of $T_k$, the space formed by gluing $k$ copies of the $(n-1)$--disc $D_1,\dots,D_k$ along their boundaries via homeomorphisms $S^{n-2}\to S^{n-2}$.
In other words, $T_k$ is homeomorphic to the join of $S^{n-2}$ and a discrete set of $k$ points.
We fix an embedding $T_k \to S^n$, and identify $T_k$ with its image.

We will see that $T_k$ separates $S^n$ into $k$ complementary components that are arranged in a natural cyclic order about their common boundary sphere $\boundary D_i \cong S^{n-2}$.

\begin{defn}[Adjacency graph]
By Alexander duality, the complement of $T_k$ in $S^n$ has $k$ components, each of which is path connected. Omitting one of the discs $D_i$ from $T_k$ yields a subspace $T^i_{k-1}$ whose complement in $S^n$ has $k-1$ components.  
The inclusion of $S^n - T_k$ into $S^n - T^i_{k-1}$ induces a surjective map of components that is one-to-one except for a pair of components $A_i$ and $B_i$ of $S^n - T_k$ that lie in the same component of $S^n - T^i_{k-1}$.
In other words, omission of the disc $D_i$ from $T_k$ has the effect of merging the two components $A_i$ and $B_i$ so that $A_i \cup B_i \cup \text{int}(D_i)$ is a component of $S^n - T^i_{k-1}$.
The \emph{adjacency graph} of complementary components of $T_k$ is the graph with one vertex for each component of $S^n - T_k$ and one edge for each $D_i$ representing the corresponding adjacency of components described above. 

In fact, we get an equivalent set of $k-1$ complementary components if we remove from $D_i$ any open $(n-1)$--ball $U$ having diameter less than $\epsilon$ such that $D_i - U$ is homeomorphic to $I \times S^{n-2}$.
Since $A_i \cup B_i \cup U$ is path connected, it contains a path from a point of $A_i$ to a point of $B_i$ that intersects $U$. Thus $U$ intersects each of $\bar{A}_i$ and $\bar{B}_i$. Since $\epsilon >0 $ is arbitrary, we see that each point of $D_i$ lies in $\bar{A}_i \cap \bar{B}_i$.
On the other hand, the argument of the preceding paragraph shows that $\text{int}(D_i)$ does not meet the closure of any component other than $A_i$ and $B_i$.
\end{defn}

If the subspace $T_k \to S^n$ were piecewise linear with respect to standard PL structures, we might interpret the adjacency graph in terms of a pair of dual cell structures on $S^n$, but this interpretation is difficult to make precise for wild topological embeddings.

\begin{defn}
The \emph{circuit} $\Circ_k$ is the graph with vertex set $\Z/k\Z$ and $k$ edges $e_1,\dots,e_k$ such that $e_i$ is adjacent to the vertices $i$ and ${i+1}$ (modulo $k$).
\end{defn}

\begin{prop}
Suppose $k \ge 2$.
For any topological embedding $T_k \to S^n$, the corresponding adjacency graph $\Gamma$ is isomorphic to the circuit $\Circ_k$. 
\end{prop}

\begin{proof}
We first claim that $\Gamma$ is a connected graph.
Indeed, if $\Gamma$ had a separation, we could partition the components of $S^n - T_k$ into two disjoint sets corresponding to the separation.  Taking the union of closures of components on each side of the separation gives a decomposition of $S^n$ as the union of two closed sets whose intersection is homeomorphic to $S^{n-2}$.  Such a decomposition is impossible by Alexander duality, since a subspace homeomorphic to $S^{n-2}$ cannot disconnect $S^n$.  Thus $\Gamma$ is connected.

We now claim that $\Gamma$ has no separating edge.
Indeed, if it did, we could decompose $S^n$ as the union of two closed sets whose intersection is a closed $(n-1)$--disc, which is again impossible by Alexander duality. Thus $\Gamma$ has no separating edge.

Note that the Euler characteristic of $\Gamma$ is zero since $\Gamma$ has an equal number of vertices and edges.
Since $\Gamma$ is a finite connected graph with $k$ vertices, no separating edges, and Euler characteristic zero, the only possibility is that it is a circuit of length $k$.
\end{proof}

\subsection{A coarse adjacency graph}

In principle, it seems reasonable to expect that the strategy above could be imitated in the coarse setting.  But to perform this conversion would require a fair amount of machinery.  Instead we replace the point-set topology arguments with entirely homological arguments.  The homological approach has the advantage that it gives the adjacency graph a more homological interpretation.  We will see that the adjacency graph is, in a certain sense, equivalent to a canonical homology class known as the Jordan cycle.

     \begin{figure}
      \begin{overpic}[scale=.5, tics=5]{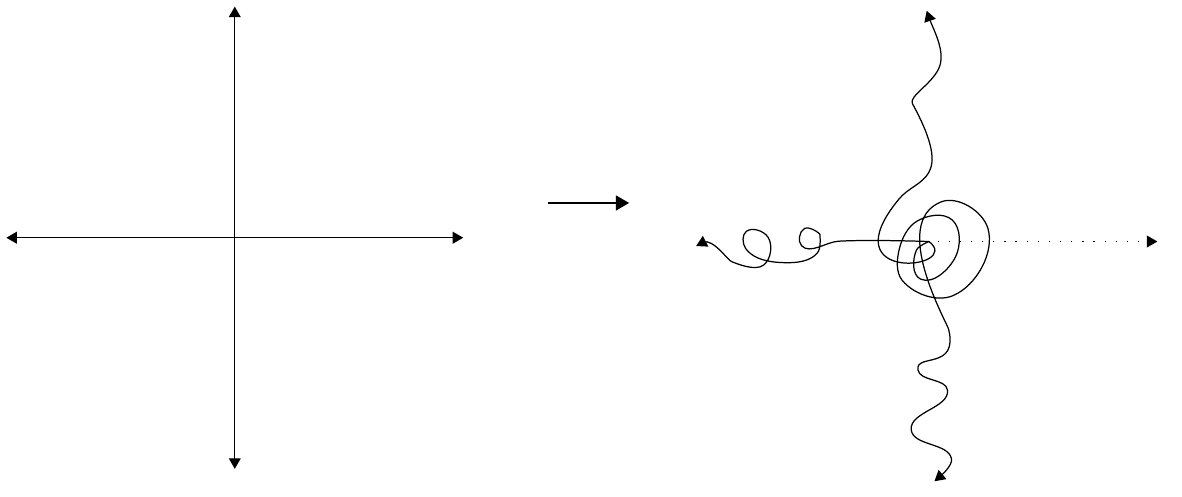}
      \put(5,38){\large{$W_4$}}
      \put(65,38){\large{$X$}}
      \put(49,27){\small{$f$}}
      \put(21,38){\small{$W^1$}}
      \put(32,23){\small{$W^3$}}
      \put(21,5){\small{$W^2$}}
      \put(3,23){\small{$W^4$}}
      \put(81,38){\small{$fW^1$}}
      \put(92,24){\small{$fW^3$}}
      \put(81,5){\small{$fW^4$}}
      \put(60,24){\small{$fW^2$}}
       \end{overpic}
	    \caption{A coarse embedding of $W_4$, a union of four $1$--dimensional halfspaces glued along their boundaries, into a coarse $PD(2)$ space $X$.}
      \label{figure:WkSetup}
     \end{figure}
    
\begin{assumption}
We use the following notation throughout this section as illustrated in Figure~\ref{figure:WkSetup}. Let $X$ be a coarse $PD(n)$ space. Let $W^i$ be a bounded geometry metric simplicial complex homeomorphic to $\R^{n-1}_{\geq 0}$ for $i \in \{1, \ldots, k\}$ with $k \geq 1$. Let
\[
   W_k = \bigcup_{i =1}^k W^i,
\]
so that the halfspaces are glued along their boundaries by simplicial homeomorphisms. Let $L_i \subset W^i$ be the boundary of $W^i$, and let $W_0$ denote the image of these boundaries in $W_k$ after gluing. Assume that $W_0$ is coarsely embedded in each halfspace $W^i$.
Let $f\colon W_k \rightarrow X$ be a coarse embedding. Let $f_\#\colon C_\circ(W_k) \rightarrow C_\circ(X)$ be an approximating chain map as given by Lemma~\ref{lem:ChainCoarseEmbedding}. 
\end{assumption}

For each $I \subset \{1,\dots,k\}$, let $W^I$ denote the union of halfspaces $W^i$ with $i \in I$.
If $I \subset J$, the inclusion $W^I \to W^J$ is a coarse embedding.
Indeed, by a standard argument, any path connecting two points of $W^I$ within $W^J$ can be replaced with a path in $W^I$ of controlled length using our assumption that $W_0 \to W^i$ is a coarse embedding for each $i$.

\begin{rem}[Cohomology calculations]
\label{rem:Calculations}
If $X$ is a coarse $PD(n)$ space, then coarse Alexander duality allows one to calculate \v{C}ech homology associated to many coarse embeddings $Z \to X$ in terms of the compactly supported cohomology of $Z$.
The cohomology of $Z$ is easily calculated for several spaces of interest in this section.
For example, it is well-known that $H^j_c(\R^\ell) \cong \Z$ if $j=\ell$ and is trivial otherwise. For halfspaces of dimension $n-1$ the group $H^{j}_c(\R^{n-1}_{\ge 0}) = 0$ for all $j$.
Similarly, $H^j_c(\R^\ell_{\ge 0}, \R^{\ell-1})= \Z$ when $j=\ell$ and is trivial otherwise.
A straightforward Mayer--Vietoris argument and induction on $k\ge 2$ shows that $H^{j}_c(W_k) = \Z^{k-1}$ if $j=n-1$ and is trivial otherwise.
\end{rem}

By the preceding calculation, the image $fW_k$ coarsely separates $X$ into $k$ deep components corresponding to the relative ends $\mathcal{E}(X,fW_k)$.
Intuitively, two relative ends may be considered adjacent if they meet along the image of a halfspace $fW^i$ for some $i \in \{1,\dots,k\}$. 

To make this idea more precise, we consider the effect of deleting a single halfspace from $W_k$.  As explained below, this deletion reduces the number of relative ends by one in a precisely controlled manner that allows us to formalize the definition of adjacency.

Consider $W_{k-1}^i := \overline{W_k - W^i}$, the subcomplex obtained by removing the $i$th halfspace from $W_k$. Whenever $2 \leq i \leq k$, coarse Alexander duality for the pair $W_{k-1}^i \hookrightarrow W_k$ gives the following commutative diagram with exact rows:
\begin{equation}
\label{eqn:C}
\begin{gathered}
\xymatrix@C-12pt{
   0 \ar[r] & \check{H}_1^D\bigl( (fW_{k-1}^i)^c,fW_k^c \bigr) \ar[r]^-{\partial} \ar[d]^{A}_{\cong} & \check{\widetilde{H}}{}_0^D(fW_k^c) \ar[r] \ar[d]^{A}_{\cong} & \check{\widetilde{H}}{}_0^D\bigl( (fW_{k-1}^i)^c \bigr) \ar[d]^{A}_{\cong} \ar[r] & 0  \\
   0 \ar[r] & H_c^{n-1}(W_k,W_{k-1}^i) \ar[r]  & H_c^{n-1}(W_k) \ar[r]  & H_c^{n-1}(W_{k-1}^i) \ar[r] & 0. \\
}
\end{gathered}
\end{equation}
Note that $H_c^j(W_k,W_{k-1}^i) = H^j_c(\R^{n-1}_{\ge 0}, \R^{n-2})$ by excision, so \eqref{eqn:C} reduces to
$0 \longrightarrow \Z \longrightarrow \Z^{k-1} \longrightarrow \Z^{k-2} \longrightarrow 0$.
By Proposition~\ref{prop_H0_gens}, the map induced by inclusion $\mathcal{E}(X,fW_k) \to \mathcal{E}(X,fW^i_{k-1})$ is surjective from a set of cardinality $k$ to a set of cardinality $k-1$.
In particular, for each $i$ there is a unique pair of distinct relative ends in $\mathcal{E}(X,fW_k)$ with the same image in $\mathcal{E}(X,fW^i_{k-1})$.
We consider these relative ends to be adjacent as explained in the following definition.

\begin{defn}[Adjacency graph] \label{def:adj_graph}
The {\it adjacency graph} $\Gamma$ has vertex set equal to the set of relative ends $V\Gamma=\mathcal{E}(X,fW_k)$.
In the degenerate case that $k=1$, the graph $\Gamma$ contains a single loop edge with both endpoints at the unique vertex.
We now assume $k\ge 2$.
For each $i$, we attach an unoriented edge $e_i$ that connects the pair of distinct vertices that have the same image under the mapping
\[
   \mathcal{E}(X,fW_k) \to \mathcal{E}(X,fW^i_{k-1}).
\]
In particular, the graph $\Gamma$ contains $k$ vertices and $k$ unoriented edges $e_1,\dots,e_k$.
\end{defn}

The edges of $\Gamma$ can be oriented as follows.
Let $v_i$ and $v'_i$ denote the vertices adjacent to $e_i$.
Since $v_i-v'_i$ generates the kernel of the map from $\check{H}_0^D(fW_k^c)$ to $\check{H}_0^D\bigl((fW_{k-1}^i)^c\bigr)$ induced by the inclusion $fW_k^c\hookrightarrow (fW_{k-1}^i)^c$, then $v_i-v'_i$ also generates the image of the homomorphism $\partial\colon \check{H}_1^D\bigl( (fW_{k-1}^i)^c,fW_k^c \bigr) \rightarrow \check{H}_0^D(fW_k^c)$. Therefore, a choice of orientation for $e_i$ is formally equivalent to a choice of generator of the cyclic group $\check{H}_1^D\bigl( (fW_{k-1}^i)^c,fW_k^c \bigr)$.

Using coarse Alexander duality, we show in this section that the adjacency graph is isomorphic to the circuit $\text{Circ}_k$.

\subsection{The adjacency chain complex}

We have given a geometric definition above for the adjacency graph.  However, in order to better understand the homology of this graph, we will study it from a more formal homological point of view.
The construction in this subsection is analogous to the formal definition of the cellular chain complex that is used to develop cellular homology in terms of singular homology groups (see \cite[\S 2.2]{hatcher}).
Similarly we use \v{C}ech homology groups to define a formal adjacency chain complex in Definition~\ref{def_chain_complex}.
In Proposition~\ref{prop:chain_complex}, we show that this chain complex is isomorphic to the standard cellular chain complex of the geometrically defined adjacency graph.

Coarse Alexander duality for the pair $W_0 \hookrightarrow W_k$ yields the following commutative diagram with exact rows 
\begin{equation}
\label{eqn:A}
\begin{gathered}
\xymatrix{
0 \ar[r]  & \check{H}_1^D(fW_0^c) \ar[r]^-{\iota} \ar[d]^{A}_{\cong} & \check{H}_1^D(fW_0^c,fW_k^c) \ar[r]^-{\partial} \ar[d]^{A}_{\cong} & \check{\widetilde{H}}{}_0^D(fW_k^c) \ar[r] \ar[d]^{A}_{\cong} & 0   \\
0 \ar[r]  & H_c^{n-2}(W_0) \ar[r]^-{\delta}  & H_c^{n-1}(W_k,W_0) \ar[r]  & H_c^{n-1}(W_k) \ar[r] & 0, \\
}
\end{gathered}
\end{equation}
which reduces to $0 \longrightarrow \Z \longrightarrow \Z^k \longrightarrow \Z^{k-1} \longrightarrow 0$.

\begin{defn} \label{def_chain_complex}
Let $\mathcal{C}(\Gamma)$ denote the cellular chain complex $\bigl\{C_i(\Gamma), \boundary \bigr\}$ of the graph $\Gamma$.
We also find it convenient to consider the following \emph{adjacency chain complex} $\mathcal{C}^a$ with $C^a_i=0$ for $i \ne 0,1$ and with $C^a_1 \to C^a_0$ equal to
\[
\xymatrix{
  \check{H}_1^D(fW_0^c, fW_k^c) \ar[r]^-{\partial} & \check{H}_0^D(fW_k^c),
}
\]
where $\partial$ denotes the composition of the connecting homomorphism $\boundary$ of \eqref{eqn:A} with the canonical inclusion $\check{\widetilde{H}}{}_0^D \to \check{H}_0^D$. 
We note that the two nontrivial terms of $\mathcal{C}^a$ are each isomorphic to $\Z^k$.
\end{defn}

The group $C^a_0 = \check{H}_0^D(fW_k^c) \cong \Z^k$ is freely generated by the set of relative ends $\cE(X, fW_k)$ by Proposition~\ref{prop_H0_gens}. Hence, this basis for $C^a_0$ corresponds to the vertex set in the adjacency graph. 
A natural basis for the group $C^a_1 = \check{H}_1^D(fW_0^c, fW_k^c)$ is given by the following notation and lemma. 

\begin{notation}
  Let
\[
   \check{H}_1^D \bigl(fW_0^c, fW_k^c\bigr) \xrightarrow[\qquad]{\iota_i} \check{H}_1^D \bigl(fW_0^c, (fW^i)^c\bigr)
\]
be the homomorphism induced by inclusion. 
\end{notation}
 
\begin{lem} \label{lemma_MViso}
The maps $\iota_1,\dots,\iota_k$ induce an isomorphism
\[
    \check{H}_1^D\bigl(fW_0^c, fW_k^c\bigr) \xrightarrow[\qquad]{I} \prod_{i=1}^k \check{H}_1^D\bigl(fW_0^c, (fW^i)^c\bigr).
\]
\end{lem}

\begin{proof}
By Proposition~\ref{prop:HomologyAxioms}(1), it suffices to show that maps induced by inclusion induce an isomorphism of cohomology with compact supports
\[
   H_c^{n-1}(W_k,W_0) \longrightarrow
   \prod_{i=1}^k H_c^{n-1} (W^i,W_0).
\]
This claim follows inductively by a straightforward Mayer--Vietoris argument.
Indeed, for each $\ell=2,\dots,k$ inclusions induce an isomorphism
\[
   H_c^{n-1}(W_\ell,W_0)
   \longrightarrow
   H_c^{n-1}(W_{\ell-1}^\ell,W_0) \times H_c^{n-1}(W^\ell,W_0),
\]
since the identity $H_c^*(W_{\ell-1}^\ell \cap W^\ell,W_0)=H_c^*(W_0,W_0)=0$ holds in all dimensions.
\end{proof}

\begin{prop}
\label{prop:chain_complex}
The chain complex $\mathcal{C}^a$ is isomorphic to the cellular chain complex $\mathcal{C}(\Gamma)$ of the adjacency graph $\Gamma$. 
\end{prop}

\begin{proof}
    The vertex set $V\Gamma = \cE(X, fW_k)$ of the adjacency graph $\Gamma$ freely generates the group $\check{H}_0^D(fW_k^c) = C_0^a$ by Proposition~\ref{prop_H0_gens}, yielding an isomorphism $C_0(\Gamma) \rightarrow C_0^a$. 
    
    By the definition of $\Gamma$, its group of $1$--chains $C_1(\Gamma)$ is isomorphic to the product $\prod_{i=1}^k \check{H}_1^D\bigl((fW_{k-1}^i)^c,fW_k^c\bigr)$, and the boundary operator $\partial\colon C_1(\Gamma)\rightarrow C_0(\Gamma)$ is induced by the maps $$\partial\colon \check{H}_1^D\bigl( (fW_{k-1}^i)^c,fW_k^c \bigr) \rightarrow \check{H}_0^D(fW_k^c)$$ for $i\in \{1,2,\cdots, k\}$.
    By our standing assumptions, we may apply the excision result of Proposition~\ref{prop:HomologyAxioms}(3) to the inclusion $(W^i, W_0) \hookrightarrow (W_k, W_{k-1}^i)$.
    Thus this inclusion induces an isomorphism
\[
    \check{H}_1^D\bigl((fW_{k-1}^i)^c, fW_k^c\bigr) \rightarrow  \check{H}_1^D\bigl(fW_0^c, (fW^i)^c\bigr),
\]
    which
    defines an isomorphism $C_1(\Gamma) \rightarrow C_1^a$ induced by inclusion via the basis for $C_1^a$ given in Lemma~\ref{lemma_MViso}. 
    
    In order to see that the two chain complexes are isomorphic, it suffices to show that following diagram commutes:
\[
\xymatrix{
  C_1(\Gamma) \ar[r]^-{\partial} \ar[d]^-{\cong} & C_0(\Gamma)  \ar[d]^-{\cong}.  \\
  C_1^a \ar[r]^-{\partial} & C_0^a.
}
\]
    It suffices to check commutativity on the basis elements of the product $C_1(\Gamma)$, which follows, for each $i$, from the commutativity of the following diagram:
\[
\xymatrix@C=10pt@R=15pt{
  & \check{H}_1^D(fW_0^c, fW_k^c) \ar[r]^-{\partial} &  \check{H}{}_0^D(fW_k^c) & & \\
 & \prod_{i=1}^k \check{H}_1^D\bigl( fW_0^c, (fW^i)^c \bigr)  \ar[u]^-{\cong}_-{I^{-1}} & & &\\
 & \check{H}_1^D\bigl( fW_0^c,(fW^i)^c \bigr)   \ar[u]_-{i_*} & & &\\
 & \check{H}_1^D\bigl( (fW_{k-1}^i)^c,fW_k^c \bigr)  \ar[u]^-{\cong}_-{\text{Excision}} \ar[r]^-{\partial} & \check{H}{}_0^D(fW_k^c) \ar[uuu]^-{\text{id}},
}
\]
where $I^{-1}$ is the inverse of the isomorphism $I$ in Lemma~\ref{lemma_MViso}, and the map in the right column is the identity map.
To see that the diagram above commutes, notice that we have inclusions of pairs
\[
   \bigl( (fW_{k-1}^i)^c,fW_k^c \bigr)
   \subset
   (fW_0^c, fW_k^c)
   \subset
   \bigl( fW_0^c,(fW^i)^c \bigr)
\]
The map $i_* I^{-1}$ is the inverse of a map induced by inclusion, so the composition of the three maps in the left column of the diagram is a map induced by inclusion. Commutativity of the diagram now follows from the naturality of the connecting homomorphism, given by Proposition~\ref{prop:HomologyAxioms}.
\end{proof}

  \subsection{The adjacency graph is a circuit}

The homology of the adjacency graph may be seen by examining the adjacency chain complex and using \eqref{eqn:A}.  One sees that the adjacency graph is connected and homotopy equivalent to a circuit; in other words, it deformation retracts onto a core graph that is isomorphic to a circuit.
In this subsection, we determine that the adjacency graph is actually isomorphic to a circuit by showing that such a deformation retraction cannot collapse any edge.
The main technique used in this proof is the Jordan cycle, a $1$--cycle that coarsely links with the set $fW_0$.
We show that the adjacency graph is, in a sense, equivalent to the Jordan cycle.

\begin{defn}
\label{def:Jordan_cycle} 
Let the {\it Jordan cycle} $\sigma$ be a generator of $\check{H}_1^D(fW_0^c) \cong \Z$.
\end{defn}

The Jordan cycle maps into the group $C_1^a = \check{H}_1^D(fW_0^c, fW_k^c)$ via \eqref{eqn:A}. In the next lemma we show that under the natural coordinate system on $\check{H}_1^D(fW_0^c, fW_k^c)$ given in Lemma~\ref{lemma_MViso} the Jordan cycle maps to the diagonal. 

\begin{lem}
\label{lem:jordan_to_diag}
  If $\sigma$ denotes the Jordan cycle and
\[
  \check{H}_1^D(fW_0^c) \xrightarrow[\qquad]{\iota} \check{H}_1^D(fW_0^c,fW_k^c)
\]
is the map induced by inclusion, then $I \circ \iota(\sigma) = (1, 1, \ldots, 1)$, for some choice of generators \textup{(}each denoted $1$\textup{)} for each infinite cyclic group $\check{H}_1^D\bigl(fW_0^c, (fW^i)^c\bigr)$.
\end{lem}

\begin{proof}
The lemma follows by examining the following commutative diagram.
\[
\xymatrix{
   \la \sigma \ra \cong \check{H}_1^D(fW_0^c) \ar[r]^-{\iota} \ar[d]^-{\cong \, \iota_A} & \check{H}_1^D(fW_0^c, fW_k^c) \ar[dl]_-{\iota_i}  \ar[d]^-{\cong \, I = (\iota_1, \ldots, \iota_k)}   \\
   \check{H}_1^D\bigl( fW_0^c, (fW^i)^c \bigr) & \prod_{i=1}^k \check{H}_1^D \bigl(fW_0^c,(fW^i)^c \bigr) \ar[l]^-{\pi_i}
}
\]
The lower right triangle commutes by the definition of $I$ in terms of the universal mapping property for direct products. The upper left  triangle commutes since the three maps are each induced by inclusion.
By Lemma~\ref{lemma_MViso}, the map $I$ is an isomorphism.  Applying Theorem~\ref{thm:CAD_Pairs} to the pair $W_0 \hookrightarrow W^i$ for $1 \leq i \leq k$ gives the exact sequence:
\begin{equation}
\label{eqn:B}
\begin{gathered}
\xymatrix{
  0 \ar[r] & \check{H}_1^D(fW_0^c) \ar[r]^-{\iota_A}  & \check{H}_1^D(fW_0^c,(fW^i)^c) \ar[r] & 0   \\
}
\end{gathered}
\end{equation}
so that the map $\iota_A$ is an isomorphism for each $i$. It follows that $I \circ \iota(\sigma) = (1, 1, \ldots, 1)$.
\end{proof}

\begin{thm} \label{thm_chain_complex_graph}
The adjacency graph is isomorphic to a circuit.
\end{thm}
\begin{proof}
The homology of the adjacency graph is equal to the homology of the adjacency chain complex $\mathcal{C}^a$ by Proposition~\ref{prop:chain_complex}, which is given by $H_0(\mathcal{C}^a)=\Z$ and $H_1(\mathcal{C}^a)=\Z$ by \eqref{eqn:A}. Therefore, the adjacency graph $\Gamma$ deformation retracts to a circuit $\Gamma_1$. 
By \eqref{eqn:A}, the infinite cyclic group $H_1(\mathcal{C}^a)$ is generated by the image of the Jordan cycle $\sigma \in \check{H}_1^D(fW_0^c)$.
In the coordinates on $C_1^a$ given by the basis in Lemma~\ref{lemma_MViso}, the Jordan cycle maps to the diagonal $(1,1, \ldots, 1)$ by Lemma~\ref{lem:jordan_to_diag}.
As described in the proof of Proposition~\ref{prop:chain_complex}, this basis for $C_1^a$ is in bijective correspondence with edges of the adjacency graph.
Therefore, each edge of $\Gamma$ also belongs to the circuit $\Gamma_1$, which implies that $\Gamma=\Gamma_1$ is a circuit.
\end{proof}

\subsection{Maps between adjacency graphs} 
\label{ss4}

\begin{notation}
    Let $1\leq i \leq k$, and let $W_{k-1}^i := \overline{W_k - W^i}$ be the subcomplex obtained by removing the $i$th halfspace from $W_k$. Let $\Gamma$ be the adjacency graph of the coarse embedding $f\colon W_k \rightarrow X$, and let $\Gamma_i$ be the adjacency graph of the restriction $f|_{W_{k-1}^i}\colon W_{k-1}^i \rightarrow X$. 
\end{notation}

\begin{thm} \label{thm:map_of_adj_graphs}
For each $k\ge 2$, the map of vertex sets induced by inclusion extends to a map of graphs $\phi\colon\Gamma\rightarrow \Gamma_i$, which collapses the circuit $\Gamma = \text{Circ}_k$ to the circuit $\Gamma_i = \text{Circ}_{k-1}$ by collapsing the $i$th edge to a point.
\end{thm}

\begin{proof}
When $k=2$ the result is trivial, so we assume $k>2$. 
For each $i$, the map $\mathcal{E}(X,fW_k) \to \mathcal{E}(X,fW^i_{k-1})$ collapses the pair of vertices corresponding to $W^i$ to a single vertex.
Whenever $i \ne j$, consider $W^{i,j}_{k-2} := \overline{W_k - (W^i \cup W^j)}$. Then the following diagram of maps induced by inclusion is commutative:
\[
\xymatrix{
   \mathcal{E}(X,fW_k) \ar[d]\ar[r] & \mathcal{E}(X,fW^i_{k-1}) \ar[d] \\
   \mathcal{E}(X,fW^j_{k-1}) \ar[r] & \mathcal{E}(X,fW^{i,j}_{k-2})
}
\]
Therefore, the map $\phi\colon \Gamma \to \Gamma_i$ sends the pair of vertices corresponding to $W^j$ to the pair of vertices corresponding to $W^j$.
\end{proof}

The geometric structure of the map between graphs given in Theorem~\ref{thm:map_of_adj_graphs} is detailed in the next lemma and will be needed in the next section. 

\begin{lem}
\label{l1}
Let $\phi\colon\Gamma \rightarrow \Gamma_i$ be the map defined in Theorem~\ref{thm:map_of_adj_graphs}. There exists a constant $R_0>0$ such that the following holds.
Suppose $u=\phi(v_i)=\phi(v'_i) \in V\Gamma_i$ for $v_i \neq v_i' \in V\Gamma$. Let $L \in \mathcal{N}(fW_{k-1}^i)$, as in Definition~\ref{defn:CechHomology}, satisfy $N_{R_0}(fW_{k-1}^i) \subseteq L$. Then, all deep components in $\pi_0 \bigl( \overline{X - L} \bigr)$ are stable, and there is a constant $M$ such that $f(W_k-N_M(W_{k-1}^i))\subset U$ where $U$ is the deep component in $\pi_0 \bigl( \overline{X - L} \bigr)$ associated to $u$.
\end{lem}

\begin{proof}
As observed in Remark~\ref{rem:Calculations}, the cohomology group $H^{n-1}_c(W_{k-1}^i)$ is free abelian of finite rank. Thus by Coarse Alexander Duality (Theorem~\ref{thm:CAD}), the dual inverse system 
$\{ H_0(\overline{X - L})\}_{L \in \mathcal{N}( fW^i_{k-1} )}$ is stable and its inverse limit is free abelian of finite rank.
By Proposition~\ref{prop_H0_gens}, there exists $R_0<\infty$ such that if $L \in \mathcal{N}(fW_{k-1}^i)$ satisfies $N_{R_0}(fW_{k-1}^i) \subseteq L$, then all deep components in $\pi_0 \bigl( \overline{X - L} \bigr)$ are stable, and the natural projection
\[
   \mathcal{E}(X,fW_{k-1}^i) \to \pi_0\bigl( \overline{X - L} \bigr)
\]
is a bijection onto the set of deep components of $\overline{X - L}$. 
Proposition~\ref{prop_H0_gens} also provides a neighborhood $L' \in \mathcal{N}(fW_{k-1}^i)$ such that any component in $\pi_0( \overline{X-L} )$ that is not deep is contained in $L'$.
Choose $M>0$ so that $f$ induces a map
\[
   \overline{ W_k - N_M(W_{k-1}^i)} \to \overline{X-L'} \to \overline{X-L}.
\]
Then $f$ maps each component in $\pi_0 (\overline{ W_k - N_M(W_{k-1}^i)})$ into a single deep component in $\pi_0\bigl( \overline{X - L} \bigr)$.
We will prove that $U$ is the deep component associated to $u$. 

Let $K \in \mathcal{N}(fW_{k})$ such that all deep components in $\pi_0 \bigl( \overline{X - K} \bigr)$ are stable and let $V$ and $V'$ be the deep components in $\pi_0 \bigl( \overline{X - K} \bigr)$ associated to $v_i$ and $v_i'$, respectively. Since $K$ is a subcomplex in $\mathcal{N}(fW_{k})$, there is a number $R_1$ such that $K\subset N_{R_1}\bigl(f(W_k)\bigr)$. Using the hypothesis $f$ is a coarse embedding, we can find a positive constant $R_2$ such that $f\bigl(N_M(W_{k-1}^i)\bigr)\subset N_{R_2}(fW_{k-1}^i)$. Let $L_1 \in \mathcal{N}(fW_{k-1}^i)$ such that $N_{R_1+R_2}(L)\subset L_1$, and let $U_1$ be the deep component in $\pi_0 \bigl( \overline{X - L_1} \bigr)$ that is associated to $u$. Therefore, $N_{R_1}(U_1)$ lies in the deep component $U'$ of $\pi_0 \bigl( \overline{X - L} \bigr)$ that is associated to $u$. We now prove that $f(W_k-N_M(W_{k-1}^i))$ has a non-empty intersection with $U'$ and this will imply that $U=U'$ is the deep component in $\pi_0 \bigl( \overline{X - L} \bigr)$ that is associated to $u$. 

Let $K_1 \in \mathcal{N}(fW_{k})$ such that $L_1\subset K_1$ and  $K\subset K_1$. Let $V_1$ and $V'_1$ be the deep components in $\pi_0 \bigl( \overline{X - K_1} \bigr)$ associated to $v_i$ and $v'_i$ respectively. Therefore, $V_1\subset V$, $V'_1\subset V'$, and both $V_1$ and $V'_1$ lie in the component $U_1$. Let $\alpha$ be a path in $U_1$ joining a point in $V_1$ to a point in $V'_1$. Therefore, $\alpha$ is a path between two deep components $V$ and $V'$ in $\pi_0 \bigl( \overline{X - K} \bigr)$, which implies that $\alpha \cap K$ is non-empty. Since $K \subset N_{R_1}(fW_k)$, there is a point $w$ in $W_k$ such that $f(w)\in N_{R_1}(\alpha)\subset N_{R_1}(U_1)\subset U'$. Also, $$f(w)\in N_{R_1}(U_1)\subset \overline{X-N_{R_2}(L)}\subset \overline{X-N_{R_2}(fW_{k-1}^i)}.$$ Thus, $w \in W_k-N_M(W_{k-1}^i)$ by the choice of $R_2$, which implies that $f(w)$ lies in $f(W_k-N_M(W_{k-1}^i))$. Therefore, the intersection $f(W_k-N_M(W_{k-1}^i))\cap U'$ is non-empty, so $U=U'$ is the deep component in $\pi_0 \bigl( \overline{X - L} \bigr)$ associated to~$u$.
\end{proof}

\begin{figure}
\begin{tikzpicture}[scale=1]

\draw (0,0) node[circle,fill,inner sep=1pt, color=black](1){} -- (0,2) node[circle,fill,inner sep=1pt, color=black](1){}-- (2,2) node[circle,fill,inner sep=1pt, color=black](1){}-- (2,0) node[circle,fill,inner sep=1pt, color=black](1){} -- (0,0) node[circle,fill,inner sep=1pt, color=black](1){}; 

\node at (-0.25,1) {\small{$e_1$}}; \node at (1,2.25) {\small{$e_2$}}; \node at (2.25,1) {\small{$e_3$}}; \node at (1,-0.25) {\small{$e_4$}};

\node at (1,-1) {$\Gamma$}; \node at (3.5,1) {$\longrightarrow$};

\draw (5,0) node[circle,fill,inner sep=1pt, color=black](1){} -- (6,2) node[circle,fill,inner sep=1pt, color=black](1){}-- (7,0) node[circle,fill,inner sep=1pt, color=black](1){}-- (5,0) node[circle,fill,inner sep=1pt, color=black](1){};

\node at (8,1) {$\longrightarrow$};

\node at (5.25,1) {\small{$e_1$}}; \node at (6.75,1) {\small{$e_3$}}; \node at (6,-0.25) {\small{$e_4$}};

\draw (10,0) node[circle,fill,inner sep=1pt, color=black](1){} -- (10.2,0.3) -- (10.3,0.6)--(10.35,0.9) --(10.35, 1.1)-- (10.3,1.4)-- (10.2,1.7)--(10,2) node[circle,fill,inner sep=1pt, color=black](1){};

\draw (10,0) node[circle,fill,inner sep=1pt, color=black](1){} -- (9.8,0.3) -- (9.7,0.6)--(9.65,0.9) --(9.65, 1.1)-- (9.7,1.4)-- (9.8,1.7)--(10,2) node[circle,fill,inner sep=1pt, color=black](1){};

\node at (9.4,1) {\small{$e_1$}}; \node at (10.6,1) {\small{$e_3$}};

\node at (10,-.25) {\small{$v$}}; \node at (10,2.25) {\small{$u$}};

\node at (10,-1) {$\Gamma_E$};

\end{tikzpicture}

\caption{Collapse of adjacency graphs}
\label{asecond}
\end{figure}

\begin{exmp}
\label{ex}
Consider the coarse embedding $f\colon W_k\rightarrow X$ for the case $k=4$, which will be used in the next section. Let $\Gamma$ be the adjacency graph of the embedding as in Figure~\ref{asecond}. Each edge $e_i$ is associated to the halfspace $W^i$. Let $E=W^1\cup W^3= \overline{W_4 - (W^2\cup W^4)}$ be the subcomplex of $W_4$ obtained by removing the second halfspace $W^2$ and the fourth halfspace $W^4$. The adjacency graph $\Gamma_E$ of the embedding $f|_{E}\colon E\rightarrow X$ is obtained as in Figure~\ref{asecond}. The relative ends $\mathcal{E}(X,fE)$ has two elements $u$ and $v$, which are used to label the vertices of $\Gamma_E$ as in Figure~\ref{asecond}. By Lemma~\ref{l1} there is a positive constant $R$ such that the following holds. For each  $L \in \mathcal{N}(fE)$ satisfying $N_{R}(fE) \subseteq L$ all deep components in $\pi_0 \bigl( \overline{X - L} \bigr)$ are stable and there is a constant $M$ such that $f(W^2-N_M(E))\subset U$ (resp. $f(W^4-N_M(E))\subset V$) where $U$ (resp. $V$) is the deep component in $\pi_0 \bigl( \overline{X - L} \bigr)$ that is associated to $u$ (resp. $v$).
\end{exmp}

\section{Amalgams with Klein bottle groups are not Kleinian} \label{sec:comm}

\begin{defn}[Isolated flats]
A $k$--flat is an isometrically embedded copy of $\E^k$ for $k\ge 2$. Let $Y$ be a $\CAT(0)$ space, and let $\text{Flat}(Y)$ denote the space of all flats in $Y$ with the topology of Hausdorff convergence on bounded sets.
A $\CAT(0)$ space $Y$ with a proper, cocompact, isometric action by a group $G$ has \emph{isolated flats} if it contains an equivariant collection $\mathcal{F}$ of flats such that $\mathcal{F}$ is closed and isolated in $\text{Flat(Y)}$ and each flat $F \subseteq X$ is contained in a uniformly bounded tubular neighborhood of some $F' \in\mathcal{F}$.
We refer the reader to \cite{HruskaKleinerIsolated} for more details on spaces with isolated flats.
\end{defn}
 
 \begin{defn}[Class of groups] \label{def:class_of_groups} 
  Let $S_g$ denote a closed orientable surface of genus $g$ greater than one. Let $T^2$ denote the $2$-torus, and let $K^2$ denote the Klein bottle.     
  Let $\cG_T$ denote the collection of amalgamated free products of the form $\pi_1(S_g)*_{\langle a=b \rangle}\pi_1(T^2)$, where $a$ is the homotopy class of an essential simple closed curve on $S_g$ and $b$ is the homotopy class of an essential simple closed curve on $T^2$.  
  Let $\cG_K$ denote the collection of amalgamated free products of the form $\pi_1(S_g)*_{\langle a=c \rangle}\pi_1(K^2)$, where $a$ is the homotopy class of an essential simple closed curve on $S_g$ and $c$ is the homotopy class of an essential simple closed curve on $K^2$ which is orientation reversing. In particular, the element $c \in \pi_1(K^2)$ acts by a glide reflection on the universal cover of $K^2$ equipped with a Euclidean metric. 
 \end{defn}

 \begin{cons}[Model geometry]  \label{const_model_geo}
  Let $G = \pi_1(S_g)*_{\la a=c \ra}\pi_1(K^2)\in \cG_K$. 
  For $\ell > 0$, there exists a hyperbolic metric on $S_g$ and a Euclidean metric on $K^2$ so the geodesic representatives of $a$ and $c$ each have length $\ell$. Glue a circle of length $\ell$ to $a$ and $c$ by isometries to form a $2$--complex $X$. Then $G \cong \pi_1(X)$.

  Observe that the universal cover $\tilde{X}$ of $X$ is a $\CAT(0)$ space with isolated flats, since it consists of flat subspaces separated by $\CAT(-1)$ regions, as discussed in \cite{kapovichleeb}.
  \end{cons}

\begin{defn}
A group $G$ is {\it Kleinian} if it is isomorphic to a discrete subgroup of $\Isom(\Hyp^3)$.
Note that for the purposes of this paper a Kleinian group may contain isometries that reverse orientation.
\end{defn}

 \begin{thm} \label{thm:not_kleinian}
  If $G \in \cG_K$, then $G$ is not a Kleinian group. 
 \end{thm}

 \begin{proof}
Let $G = \pi_1(S_h)*_{\langle a=c \rangle}\pi_1(K^2)$, where $a$ is the homotopy class of an essential simple closed curve on $S_h$ and $c$ is the homotopy class of an essential simple closed curve on $K^2$ which is orientation reversing. Then, $G \cong \pi_1(X)$ as defined in Construction~\ref{const_model_geo}. Suppose towards a contradiction that $G$ is a Kleinian group. 

Let $P = \pi_1(K^2) \leq G$. Since  $P$ has a finite-index subgroup isomorphic to $\Z^2$, each conjugate of $P$ in $G$ acts parabolically on $\Hyp^3$.
Since the universal cover $\widetilde{X}$ has isolated flats, the Hausdorff distance between any two flat subspaces of $\widetilde{X}$ is infinite.  In particular, the stabilizers of distinct flats are not commensurable.
It follows that each conjugate of $P$ in $G$ fixes a distinct parabolic point in $\partial \Hyp^3$, the visual boundary of $\Hyp^3$. Indeed, if $p \in \p \Hyp^3$ is stabilized by a parabolic isometry of $G$, then $\Stab_G(p)$ does not contain any loxodromic elements since $G$ acts discretely on $\Hyp^3$. So, $\Stab_G(p)$ is a virtually abelian group of rank at most two. In particular, if $P$ and $gPg^{-1}$ for $g \in G$ fix the same parabolic point $p$, then $P$ and $gPg^{-1}$ are commensurable in $\Stab_G(p)$, a contradiction.
    
We will show there exists a horosphere $F$ stabilized by $P$ that does not intersect any of its $G$--translates. 
Each conjugate of $P$ in $G$ stabilizes a family of horospheres centered at a parabolic point on $\p\Hyp^3$. Since $G$ acts on $\Hyp^3$ discretely and $P$ acts on each horosphere cocompactly, if $F$ is any horosphere stabilized by $P$, then there exists a compact subset $C \subset F$ so that $P \cdot C$ covers $F$. Any horosphere $gF$ for $g \in G$ satisfying $gF \cap F \neq \emptyset$ can be translated using $P$ to a horosphere stabilized by a conjugate of $P$ that intersects $C$. Since $G$ acts on $\Hyp^3$ properly, there are only finitely many horospheres stabilized by a conjugate of $P$ which interset $C$. Thus, the horosphere $F$ can be shrunk so that it is disjoint from all of its $G$--translates. Moreover, all other $G$--translates of $F$ lie in one component of $\Hyp^3 \setminus F$.
    
     \begin{figure}
      \begin{overpic}[scale=.8, tics=5]{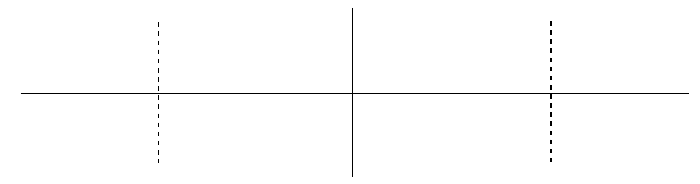}
    	\put(48,26.3){\small{$E_1$}}
    	\put(48,-3){\small{$E_2$}}
    	\put(-2,12){\small{$H_1$}}
    	\put(100,12){\small{$H_2$}}
      	\put(21,24){\small{$\tilde K_1$}}
    	\put(77,24){\small{$\tilde K_2$}}
      \end{overpic}
	    \caption{In $\widetilde X$, a Euclidean plane $E_1 \cup E_2$ intersects a hyperbolic plane $H_1 \cup H_2$ in a line. The image depicts the cyclic order of the image of these half-spaces under a $G$--equivariant coarse embedding to $\Hyp^3$. If Euclidean planes $\widetilde K_1$ and $\widetilde K_2$ are chosen very far from the plane $E_1 \cup E_2$, they map to opposite sides of the horospherical image of the plane $E_1\cup E_2$ in $\Hyp^3$, a contradiction.}
      \label{figure:planes}
     \end{figure}
    
To reach a contradiction, we use coarse separation results and the action of $G$ on $\widetilde{X}$ to prove the horosphere $F$ separates two of its $G$--translates. Let $\ell \subset \widetilde{X}$ be the line stabilized by $\la a \ra$. The line $\ell$ is contained in $\widetilde{S}_h$, a copy of the universal cover of $S_h$, and $\widetilde{K}_0$, a copy of the universal cover of the Klein bottle. Let $H_1$ and $H_2$ be the half-planes in $\widetilde{S}_h$ bounded by $\ell$, and let $E_1$ and $E_2$ be the half-planes in $\widetilde{K}_0$ bounded by $\ell$. Let $W = H_1 \cup H_2 \cup E_1 \cup E_2$.
    
Since $G=\pi_1(X)$ acts properly by isometries on $\Hyp^3$, there is a $G$--equivariant coarse embedding $f\colon \widetilde{X} \rightarrow \Hyp^3$. 
There is a constant $A$ such that for each lift $\widetilde{K}$ of the Klein bottle in $\widetilde{X}$ there is a $G$--translate $gF$ in $\Hyp^3$ such that the Hausdorff distance between $f(\widetilde{K})$ and $gF$ is less than $A$. Let $\Gamma$ be the adjacency graph of the coarse embedding $f|_{W}\colon W\rightarrow \Hyp^3$ as defined in Section~\ref{sec:Jordan_adj_graph}. The graph $\Gamma$ is a $4$--cycle whose edges are associated to (and labeled by) $H_1$, $H_2$, $E_1$, and $E_2$. 
Since the action of $a$ on $\widetilde{X}$ stabilizes $W$---permuting its four halfplanes---and the coarse embedding $W \to \Hyp^3$ is $a$--equivariant, the action of $a$ induces a graph automorphism of $\Gamma$.

Moreover, the action on $\Gamma$ stabilizes each edge labeled by $H_i$ and interchanges edges labeled by $E_1$ and $E_2$. Therefore, two edges of the $4$--cycle $\Gamma$ labelled by $H_1$ and $H_2$ are opposite. 
    
Recall that $\widetilde{K}_0=E_1\cup E_2$. By Example~\ref{ex} and the fact the Hausdorff distance between $f\widetilde{K}_0$ and $F$ is finite, the relative end $\mathcal{E}(\Hyp^3,F)$ consists of two elements. Moreover, there exists a positive constant $R$ such that for each  $L \in \mathcal{N}(F)$ satisfying $N_{R}(F) \subseteq L$, all deep components in $\pi_0 \bigl( \overline{\Hyp^3 - L} \bigr)$ are stable, and there is a constant $M$ such that $f(H_i-N_M( \widetilde{K}_0))\subset U_i$, where $U_1$  and $U_2$ are deep components in $\pi_0 \bigl( \overline{\Hyp^3 - L} \bigr)$, each associated to one of the two elements of $\mathcal{E}(\Hyp^3,F)$.

As explained in Example~\ref{ex}, the adjacency graph $4$--cycle associated to $W\to \Hyp^3$ collapses onto the adjacency graph $2$--cycle associated to $\widetilde{K}_0$ by a map that collapses each edge $H_i$ to a point. In particular, the two vertices of $\Gamma$ adjacent to $H_1$ map to the same side of $\widetilde{K}_0$, which is different from the side containing the two vertices adjacent to $H_2$.  In particular, $U_1\ne U_2$.

    Let $V_1$ and $V_2$ be the two components of $\Hyp^3-F$. Assume that $V_i$ is associated to the same relative end as $U_i$ for each $i$. Let $L$ be a subcomplex in $\mathcal{N}(F)$ such that $N_{R+A}(F) \subseteq L$. Observe that $N_A(U_i)\subset V_i$ for each $i$. By the above discussion, there is a positive number $M$ such that $f(H_i-N_M(\widetilde{K}_0))\subset U_i$ for each $i$. Let $\widetilde K_1$ and $\widetilde K_2$ be two lifts of the Klein bottle in $\widetilde{X}$ such that each intersection $\widetilde K_i\cap H_i$ contains a point not in $N_M(\widetilde{K}_0)$. Therefore, $f \widetilde K_i$ intersects with $U_i$ for each $i$. Let $g_1 F$ and $g_2 F$ be two $G$--translates of $F$ such that the Hausdorff distance between $g_i F$ and $f \widetilde K_i$ is less than $A$ for each $i$. Therefore, each intersection $g_i F\cap V_i$ is non-empty which implies that $g_i F\subset V_i$ for each $i$. In other words, the horosphere $F$ separates its $G$--translates $g_1 F$ and $g_2 F$ which is a contradiction. Therefore, $G$ is not a Kleinian group.  
\end{proof}
 
\begin{prop}
\label{prop:Accidental}
Let $S_g$ be a closed orientable surface of genus $g\ge 2$, and let $a$ be the homotopy class of an essential simple closed curve in $S_g$.
Then the amalgamated product
\[
   G = \pi_1(S_g) *_{\langle a \rangle} \bigl(\langle a \rangle \times \Z \bigr).
\]
is isomorphic to a geometrically finite Kleinian group.
\end{prop}

\begin{proof}
The group $G$ is the fundamental group of a compact $3$--manifold $M$ formed from the product $S_g \times [-1,1]$ by removing the interior of a closed tubular neighborhood $N$ of the curve $a_0 \times \{0\}$, where $a_0$ represents the homotopy class $a$.
Since $(M,\boundary N)$ is a compact atoroidal Haken pared $3$--manifold, the result follows directly from Thurston's Hyperbolization Theorem (see \cite[Thm.~1.43]{kapovich}).
\end{proof}

\begin{prop}
\label{prop:degree2_cover}
The group $G = \pi_1(S_g)*_{\la a = c \ra} \pi_1(K^2) \in \cG_K$ where $a$ is the homotopy class of a nonseparating simple closed curve on $S_g$ has an index two subgroup that is a geometrically finite Kleinian group.
\end{prop}

\begin{proof}
Suppose $G = \pi_1(S_g)*_{\la a = c \ra} \pi_1(K^2) \in \cG_K$, where $a$ is the homotopy class of a nonseparating simple closed curve on $S_g$. Then $G \cong \pi_1(X)$, where $X$ is the union of $S_g$ and $K^2$ glued along essential simple closed curves $a_0$ and $c_0$ on $S_g$ and $K^2$ respectively. The curve $a_0$ is a nonseparating simple closed curve on $S_g$ in the homotopy class of $a \in \pi_1(S_g)$, and $c_0$ is an orientation-reversing simple closed curve on $K^2$ in the homotopy class of $c \in \pi_1(K^2)$. 

To prove the proposition, we show that $X$ has a double cover $\hat{X}$ formed from a closed orientable surface $S_{2g-1}$ and a torus $T^2$ by identifying an essential simple closed curve from each surface.  
Since the loop $c_0$ reverses orientation, its preimage in the canonical orientable double cover $T^2 \to K^2$ is a single simple closed curve $\hat{c}$ double covering $c_0$. Since the simple closed curve $a_0$ does not separate $S_g$, it represents a nontrivial element of $H_1(S_g; \Z/2\Z) = (\Z/2\Z)^{2g}$.
Thus there exists a homomorphism
\[
   \pi_1(S_g) \longrightarrow H_1(S_g;\Z) \longrightarrow H_1(S_g;\Z/2\Z) \longrightarrow \Z /2\Z
\]
mapping $a$ to the non-trivial element of $\Z /2\Z$. In the double cover $S_{2g-1} \rightarrow S_g$, the curve $a_0$ has preimage a single simple closed curve $\hat{a}$ double covering $a_0$. Form a $2$--complex $\hat{X}$ from $T^2$ and $S_{2g-1}$ by identifying $\hat{c}$ with $\hat{a}$. Then $\hat{X}$ double covers $X$.
By Proposition~\ref{prop:Accidental} the group $\pi_1(\hat{X})$ is isomorphic to a geometrically finite Kleinian group.
\end{proof}

\begin{rem}
\label{rem:virt_Kleinian}
If $G = \pi_1(S_g)*_{\la a = c \ra} \pi_1(K^2) \in \cG_K$ where $a$ is the homotopy class of a separating simple closed curve on $S_g$, then $G$ has an index four subgroup that is geometrically finite and Kleinian. The proof above does not apply since if $a$ is separating, it has trivial image in the abelianization $H_1(S_g)$ of $\pi_1(S_g)$, and hence also in the abelian group $\Z/2\Z$.
In particular, $a$ is contained in every index two subgroup of $\pi_1(S_g)$, and any loop representing $a$ has a disconnected preimage in every double cover of $S_g$.
However, $G$ and $G' = \pi_1(S_g)*_{\la a = b \ra} \pi_1(T^2) \in \cG_K$ do have isomorphic index four subgroups, and the group $G'$ is a geometrically finite Kleinian group. We leave the details as an exercise for the reader. 
\end{rem}

\bibliographystyle{alpha}
\bibliography{refs}

\begin{thebibliography}{BKK02}

\bibitem[Bes96]{Bestvina_LocalHomology}
Mladen Bestvina.
\newblock Local homology properties of boundaries of groups.
\newblock {\em Michigan Math. J.}, 43(1):123--139, 1996.

\bibitem[BKK02]{bestvinakapovichkleiner}
Mladen Bestvina, Michael Kapovich, and Bruce Kleiner.
\newblock Van {K}ampen's embedding obstruction for discrete groups.
\newblock {\em Invent. Math.}, 150(2):219--235, 2002.

\bibitem[Bro82]{Brown_Cohomology}
Kenneth~S. Brown.
\newblock {\em Cohomology of groups}, volume~87 of {\em Graduate Texts in
  Mathematics}.
\newblock Springer-Verlag, New York-Berlin, 1982.

\bibitem[BW97]{BlockWeinberger97}
Jonathan Block and Shmuel Weinberger.
\newblock Large scale homology theories and geometry.
\newblock In William~H. Kazez, editor, {\em Geometric topology
  \textup{(}{A}thens, {GA}, 1993\textup{)}}, volume~2 of {\em AMS/IP Stud. Adv.
  Math.}, pages 522--569. Amer. Math. Soc., Providence, RI, 1997.

\bibitem[Dyd81]{dydak}
Jerzy Dydak.
\newblock Local {$n$}--connectivity of quotient spaces and one-point
  compactifications.
\newblock In S.~Marde\v{s}i\'{c} and J.~Segal, editors, {\em Shape theory and
  geometric topology \textup{(}{D}ubrovnik, 1981\textup{)}}, volume 870 of {\em
  Lecture Notes in Math.}, pages 48--72. Springer, Berlin-New York, 1981.

\bibitem[ES52]{EilenbergSteenrod_Axioms}
Samuel Eilenberg and Norman Steenrod.
\newblock {\em Foundations of algebraic topology}.
\newblock Princeton University Press, Princeton, New Jersey, 1952.

\bibitem[FS96]{FarbSchwartz96}
Benson Farb and Richard Schwartz.
\newblock The large-scale geometry of {H}ilbert modular groups.
\newblock {\em J. Differential Geom.}, 44(3):435--478, 1996.

\bibitem[Geo86]{Geoghegan86_ShapeOfAGroup}
Ross Geoghegan.
\newblock The shape of a group---connections between shape theory and the
  homology of groups.
\newblock In H.~Toru\'{n}czyk, S.~Jackowski, and S.~Spie\.{z}, editors, {\em
  Geometric and algebraic topology}, volume~18 of {\em Banach Center Publ.},
  pages 271--280. PWN, Warsaw, 1986.

\bibitem[Geo08]{geoghegan}
Ross Geoghegan.
\newblock {\em Topological methods in group theory}, volume 243 of {\em
  Graduate Texts in Mathematics}.
\newblock Springer, New York, 2008.

\bibitem[GS19]{GeogheganSwenson_Semistable}
Ross Geoghegan and Eric Swenson.
\newblock On semistability of {$\rm CAT(0)$} groups.
\newblock {\em Groups Geom. Dyn.}, 13(2):695--705, 2019.

\bibitem[Gui16]{guilbault}
Craig~R. Guilbault.
\newblock Ends, shapes, and boundaries in manifold topology and geometric group
  theory.
\newblock In M.W. Davis, J.~Fowler, J.-F. Lafont, and I.J. Leary, editors, {\em
  Topology and geometric group theory}, volume 184 of {\em Springer Proc. Math.
  Stat.}, pages 45--125. Springer, [Cham], 2016.

\bibitem[Hat02]{hatcher}
Allen Hatcher.
\newblock {\em Algebraic topology}.
\newblock Cambridge University Press, Cambridge, 2002.

\bibitem[HK05]{HruskaKleinerIsolated}
G.~Christopher Hruska and B.~Kleiner.
\newblock Hadamard spaces with isolated flats.
\newblock {\em Geom. Topol.}, 9:1501--1538, 2005.

\bibitem[HST20]{hruskastarktran}
G.~Christopher Hruska, Emily Stark, and Hung~Cong Tran.
\newblock Surface group amalgams that (don't) act on $3$--manifolds.
\newblock {\em Amer. J. Math.}, 142(3):885--921, 2020.

\bibitem[Kap01]{kapovich}
Michael Kapovich.
\newblock {\em Hyperbolic manifolds and discrete groups}.
\newblock Modern Birkh\"auser Classics. Birkh\"auser Boston, Inc., Boston, MA,
  2001.

\bibitem[KK00]{kapovichkleiner}
Michael Kapovich and Bruce Kleiner.
\newblock Hyperbolic groups with low-dimensional boundary.
\newblock {\em Ann. Sci. \'Ecole Norm. Sup. \textup{(}4\textup{)}},
  33(5):647--669, 2000.

\bibitem[KK05]{kapovichkleiner05}
Michael Kapovich and Bruce Kleiner.
\newblock Coarse {A}lexander duality and duality groups.
\newblock {\em J. Differential Geom.}, 69(2):279--352, 2005.

\bibitem[KL95]{kapovichleeb}
M.~Kapovich and B.~Leeb.
\newblock On asymptotic cones and quasi-isometry classes of fundamental groups
  of {$3$}--manifolds.
\newblock {\em Geom. Funct. Anal.}, 5(3):582--603, 1995.

\bibitem[KL97]{KapovichLeeb_QI}
Michael Kapovich and Bernhard Leeb.
\newblock Quasi-isometries preserve the geometric decomposition of {H}aken
  manifolds.
\newblock {\em Invent. Math.}, 128(2):393--416, 1997.

\bibitem[Mil95]{Milnor_Steenrod}
John Milnor.
\newblock On the {S}teenrod homology theory.
\newblock In S.C. Ferry, A.~Ranicki, and J.~Rosenberg, editors, {\em Novikov
  conjectures, index theorems and rigidity, {V}ol. 1 \textup{(}{O}berwolfach,
  1993\textup{)}}, volume 226 of {\em London Math. Soc. Lecture Note Ser.},
  pages 79--96. Cambridge Univ. Press, Cambridge, 1995.

\bibitem[MS82]{MardesicSegal82}
Sibe Marde\v{s}i\'{c} and Jack Segal.
\newblock {\em Shape theory: {T}he inverse system approach}, volume~26 of {\em
  North-Holland Mathematical Library}.
\newblock North-Holland Publishing Co., Amsterdam-New York, 1982.

\bibitem[NS97]{NeumannSwarup_Canonical}
Walter~D. Neumann and Gadde~A. Swarup.
\newblock Canonical decompositions of {$3$}--manifolds.
\newblock {\em Geom. Topol.}, 1:21--40, 1997.

\bibitem[Roe03]{Roe_CoarseGeometry}
John Roe.
\newblock {\em Lectures on coarse geometry}, volume~31 of {\em University
  Lecture Series}.
\newblock American Mathematical Society, Providence, RI, 2003.

\bibitem[Spa66]{Spanier_AlgebraicTop}
Edwin~H. Spanier.
\newblock {\em Algebraic topology}.
\newblock McGraw-Hill Book Co., New York-Toronto, Ont.-London, 1966.

\end{thebibliography}

\end{document}